\documentclass[12pt]{amsart}
\usepackage[top=80pt,bottom=80pt,left=70pt,right=70pt]{geometry}
\usepackage{graphicx} % Required for inserting images
\usepackage{amsmath}
\usepackage{enumitem}
\usepackage{amssymb}
\usepackage{amsthm}
\usepackage{hyperref}
\usepackage{tikz}
\usepackage{tikz-cd} 

\usepackage[utf8]{inputenc}

\title[actions with globally hypoelliptic orbitwise laplacian]{Classification of abelian actions with globally hypoelliptic orbitwise laplacian I: The Greenfield-Wallach conjecture on nilmanifolds}
\author[S.~Sandfeldt]{Sven Sandfeldt}
\address{Department of mathematics, Kungliga Tekniska högskolan, Lindstedtsvägen 25, SE-100 44 Stockholm, Sweden.}
\email{svensan@kth.se}
\date{\today}

\newcommand{\norm}[1]{\left\lVert#1\right\rVert}

\newcommand{\intd}{{\rm d}}

\makeatletter
\newtheorem*{rep@theorem}{\rep@title}
\newcommand{\newreptheorem}[2]{%
\newenvironment{rep#1}[1]{%
 \def\rep@title{#2 \ref{##1}}%
 \begin{rep@theorem}}%
 {\end{rep@theorem}}}
\makeatother

\counterwithin{equation}{section}

\newtheorem{mainTheorem}{Theorem}

\newreptheorem{mainTheorem}{Theorem}

\newreptheorem{mainCorollary}{Corollary}

\newtheorem{theorem}{Theorem}[section]

\newtheorem{lemma}{Lemma}[section]

\newtheorem{definition}{Definition}[section]
\newtheorem*{definition*}{Definition}

\theoremstyle{definition}
\newtheorem{example}{Example}[section]
\newtheorem{remark}{Remark}

\begin{document}

\begin{abstract}
For a $\mathbb{R}^{k}-$action generated by vector fields $X_{1},...,X_{k}$ we define an operator $-(X_{1}^{2}+...+X_{k}^{2})$, the orbitwise laplacian. In this paper, we study and classify $\mathbb{R}^{k}-$actions whose orbitwise laplacian is globally hypoelliptic (GH). In three different settings we prove that any such action is given by a translation action on some compact nilmanifold, (i) when the space is a compact nilmanifold, (ii) when the first Betti number of the manifold is sufficiently large, (iii) when the codimension of the orbitfoliation of the action is $1$. As a consequence, we prove the Greenfield-Wallach conjecture on all nilmanifolds. Along the way, we also calculate the cohomology of GH $\mathbb{R}^{k}-$actions, proving, in particular, that it is always finite dimensional.
\end{abstract}

\maketitle

\section{Introduction}

Let $M$ be a closed smooth $d-$manifold and $\Omega^{\ell}(M)$ its space of $\ell-$forms. An operator $L:C^{\infty}(M)\to C^{\infty}(M)$ is said to be \textit{globally hypoelliptic} if any distributional solution $D\in C^{\infty}(M)'$ to the equation
\begin{align}
L'D = \omega,\quad\omega\in\Omega^{d}(M)
\end{align}
satisfy $D\in\Omega^{d}(M)$. That is, if $L'D = \omega$ is a distributional solution with smooth data $\omega$, then $D$ is also smooth. If $X$ is a vector field on $M$ then $X$ naturally defines a first order differential operator by differentiating along $X$. A vector field is \textit{globally hypoelliptic} if the associated differential operator is globally hypoelliptic. A vector $\mathbf{v}\in\mathbb{R}^{d}$ is \textit{diophantine} if $|\mathbf{v}\cdot\mathbf{n}|\geq C/\norm{\mathbf{n}}^{\tau}$, $\mathbf{n}\in\mathbb{Z}^{d}\setminus0$, for some constants $C$ and $\tau$. It is a well-known fact that any diophantine vector $\mathbf{v}$, considered as a vector field on the torus $\mathbb{T}^{d}$, is a globally hypoelliptic \cite[Section 3]{GreenfieldWallach2}. In \cite{GreenfieldWallach1,GreenfieldWallach2} (\cite[Problem 2]{GreenfieldWallach1}) S. Greenfield and N. Wallach conjectured that the converse also holds: every GH vector field is (up to a smooth change of coordinates) given by a constant diophantine vector field on a torus. This is known as the \textit{Greenfield-Wallach conjecture}. Early progress was made in the original paper by Greenfield and Wallach, where they gave a complete solution when the manifold is $2-$dimensional \cite[Theorem 1.4]{GreenfieldWallach1}, a complete solution for homogeneous vector fields on $3-$manifolds \cite[Proposition 2.1]{GreenfieldWallach1}, and a complete solution for killing vector fields on any Riemannian manifold \cite[Proposition 2.2]{GreenfieldWallach1}. 

Given a vector field $X$ we refer to the equation
\begin{align}
Xu = v + c,\quad u,v\in C^{\infty}(M),\text{ }c\in\mathbb{C},
\end{align}
as the (smooth) \textit{cohomological equation}. The cohomological equation (and its various twisted versions) is important in many problems arising from dynamics. Notably, understanding the cohomological equation was key in the approach initiated by G. Forni and L. Flaminio to study effective equidistribution of various parabolic flows \cite{FlaminioForni2003,FlaminioForni2006,Kim2022,FlaminioForni2023}. Another problem in which (twisted) cohomological equations are important, is in the rigidity of Anosov actions. Both the topological rigidity of Anosov diffeomorphisms on (infra-)nilmanifolds \cite{Franks1969,Manning1974}, and the smooth rigidity of higher rank Anosov actions on (infra-)nilmanifolds \cite{FisherKalininSpatzier2013,Rodriguez-HertzWang2014}, are proved by studying certain naturally occurring cohomological equations (in fact, these are the same cohomological equations as those that occur in Section \ref{SubSec:GWConjectureOnNilmanifolds}). The cohomological equation also plays a key role in the \textit{KAM approach} to local rigidity of actions \cite{DamjanovicKatok2010,DamjanovicKatok2011Parabolic,DamjanovicFayadSaprykina2023,DamjanovicFayad2019,Wang2019,Wang2022}. We say that the vector field $X$ is \textit{cohomology free} if the cohomological equation $Xu = v + c$ has a solution $u\in C^{\infty}(M)$, $c\in\mathbb{C}$ for every $v\in C^{\infty}(M)$. It is a well-known fact that the diophantine constant vectors $\mathbf{v}\in\mathbb{R}^{d}$ define cohomology free vector fields on $\mathbb{T}^{d}$. A conjecture due to Katok \cite{KatokCombinatorial,Hurder}, \textit{the Katok conjecture}, states that the converse also holds: a vector field $X$ is cohomology free if and only if $X$ is (up to smooth coordinate change) a constant diophantine vector field on a torus. It is not hard to check that if $X$ is cohomology free then $X$ is also globally hypoelliptic (see \cite[Proposition 3.6]{ForniGLWdim3}). As a consequence, the Greenfield-Wallach conjecture implies the Katok conjecture. In \cite{ForniGLWdim3} (building on work in \cite{ChenChiEqGW}) it was shown that the converse also holds: if $X$ is a globally hypoelliptic vector field, then $X$ is cohomology free.

Progress towards the Greenfield-Wallach conjecture (and the Katok conjecture) has been made by making two distinct assumptions on the system under consideration. Either one makes topological assumptions on the manifold, or one makes algebraic assumptions on the manifold (and possibly also the dynamics). As an example, if one assumes that the manifold $M\cong\mathbb{T}^{d}$ is a torus, then the Greenfield-Wallach conjecture is known to hold \cite{ChenChiEqGW}. That is, the Greenfield-Wallach conjecture holds if one assumes that the space is the one predicted by the conjecture. A stronger result was proved by F. Rodriguez Hertz and J. Rodriguez Hertz in \cite{RordriguezHertzRodriguezHertz2006} where the authors show that if $X$ is a cohomology free vector field on $M$ and $M$ has first betti-number $b_{1}(M)$, then there is a submersion $M\to\mathbb{T}^{b_{1}}$ conjugating $X$ onto a linear flow. In particular, if the first betti number of $M$ is large enough, then the Greenfield-Wallach conjecture holds on $M$. Similar ideas were also used in \cite{FlaminioPaternainEmbedding} to show that any cohomology free flow embeds as a linear flow in some abelian group. The results in \cite{RordriguezHertzRodriguezHertz2006}, and the solution to the Weinstein conjecture \cite{TaubesWeinstein}, were later used independently by G. Forni and A. Kocsard to prove the Greenfield-Wallach conjecture completely in dimension $3$ \cite{KocsardGLWdim3,ForniGLWdim3}. Progress towards the Greenfield-Wallach conjecture has also been made when $M$ is a homogeneous space, and the vector field $X$ is homogeneous. Indeed, in \cite{ForniRHGWC} it is shown that a homogeneous vector field $X$ can only be cohomology free if $M$ is a torus, so in this case the Greenfield-Wallach conjecture holds. In particular, any globally hypoelliptic homogeneous vector field on a nilmanifold $M$ is a constant diophantine vector field on a torus. Our first main theorem extends this result from all homogeneous vector fields to \textit{all vector fields}.
\begin{mainTheorem}\label{MainThm:ThmA}
The Greenfield-Wallach conjecture is true on all compact nilmanifolds. That is, the only nilmanifold that supports a globally hypoelliptic vector field is the torus.
\end{mainTheorem}
Let $\alpha:\mathbb{R}^{k}\times M\to M$ be a smooth $\mathbb{R}^{k}-$action generated by vector fields $X_{1},...,X_{k}$. Define the \textit{orbitwise laplacian} of $\alpha$ on $C^{\infty}(M)$ by
\begin{align*}
\Delta_{\alpha}:C^{\infty}(M)\to C^{\infty}(M),\quad\Delta_{\alpha}u := -\left(X_{1}^{2}+...+X_{k}^{2}\right)u.
\end{align*}
We say that $\alpha$ is globally hypoelliptic (GH) if $\Delta_{\alpha}$ is globally hypoelliptic. If $k = 1$, then $\Delta_{\alpha}$ is GH if and only if the vector field $X$ generating $\alpha$ is globally hypoelliptic. So, the Greenfield-Wallach conjecture can be stated as: a $\mathbb{R}-$action $\alpha$ is GH if and only if $\alpha$ is smoothly conjugated to a diophantine linear flow on a torus. In \cite{DanijelaGH2} D. Damjanović, J. Tanis, and Z. Wang posed a conjecture for $\mathbb{R}^{k}-$actions, stating that any homogeneous GH action on a finite volume homogeneous space $G/D$ is smoothly conjugate to a translation action on some nilmanifold (of step $\ell\leq k$). Since a $\mathbb{R}-$action is GH if and only if the generating vector field $X$ is globally hypoelliptic, this extends the Greenfield-Wallach conjecture in the context of homogeneous actions. A related conjecture has also been raised by J. Cygan and L. Richardson for families of vector fields \cite{CyganRichardson1988,CyganRichardson1991}.

The main result of this paper is that certain GH actions are smoothly conjugated to translation actions on nilmanifolds. Recall that the first Betti number of a manifold $M$ is the dimension of the first real homology group, $b_{1}(M) = \dim(H_{1}(M))$. 
\begin{mainTheorem}\label{MainThm:ThmB}
Let $\alpha:\mathbb{R}^{k}\times M\to M$ be a locally free ${\rm GH}$ action. If one of the following holds:
\begin{enumerate}[label = (\roman*)]
    \item $M$ is a nilmanifold,
    \item $b_{1}(M)\geq\dim(M) - 1$, where $b_{1}(M)$ is the first Betti number,
    \item $k\geq\dim(M) - 1$, 
\end{enumerate}
then $\alpha$ is smoothly conjugated to an action by translations on some nilmanifold $M_{\Gamma} = \Gamma\setminus G$. Moreover, in case $(ii)$ the nilpotent Lie group is $G = H^{g}\times\mathbb{R}^{n}$ where $H^{g}$ is the $g$th Heisenberg group (and $H^{0} = 1$), and in case $(iii)$ the nilpotent Lie group is either $G = H^{1}\times\mathbb{R}^{n}$ or $G = \mathbb{R}^{n}$.
\end{mainTheorem}
\begin{remark}
For a definition of the Heisenberg groups, see Example \ref{Ex:HeisenberGroup}.
\end{remark}
\begin{remark}
In \cite{DanijelaGH2} the authors conjecture that all homogeneous GH actions are translation actions on nilmanifolds. Motivated by Theorem \ref{MainThm:ThmB} it seems likely that more is true, and that the following question has an affirmative answer: is every GH $\mathbb{R}^{k}-$action given by a translation action on a nilmanifold?
\end{remark}
\begin{remark}
The conclusion of Theorem \ref{MainThm:ThmB} is that the action $\alpha$ is a translation action on a nilmanifold, not a torus. The existence of such actions on $2-$step nilmanifolds was shown by D. Damjanović in \cite{DanijelaGH1}. In forthcoming work \cite{Sandfeldt2024}, we construct a rich class of GH translation actions on higher step nilmanifolds, we also conjecture that this class is complete in the sense that every GH $\mathbb{R}^{k}-$action is smoothly conjugated to an action within the class.
\end{remark}
\begin{remark}
The Katok conjecture can also be stated for diffeomorphisms $f:M\to M$, instead of for flows. In this case, the cohomological equation takes the form
\begin{align}\label{Eq:DiscreteCohomologicalEquation}
u(fx) - u(x) = v(x) + c,\quad u,v\in C^{\infty}(M),\text{ }c\in\mathbb{C},
\end{align}
and the Katok conjecture states that Equation \ref{Eq:DiscreteCohomologicalEquation} has a solution for every $v\in C^{\infty}(M)$ if and only if $f$ is (up to smooth coordinate change) a diophantine translation on a torus. Much less is known about the discrete Katok conjecture. Indeed, apart from some results on low-dimensional tori \cite{dosSantosLuz1998}, nothing is known in full generality. If one restricts to diffeomorphisms homotopic to identity, then more can be proved and, in particular, the results from \cite{RordriguezHertzRodriguezHertz2006} still apply. For a $\mathbb{Z}^{k}-$action $\alpha$ generated by $f_{1},...,f_{k}$, we define a (discrete) orbitwise laplacian as
\begin{align}
\Delta_{\alpha}u = -\sum_{j = 1}^{k}\left(u\circ f_{j} - 2u + u\circ f_{j}^{-1}\right).
\end{align}
We say that $\alpha$ is GH if $\Delta_{\alpha}$ is globally hypoelliptic. Many of the results of this paper can be generalized to GH $\mathbb{Z}^{k}-$actions that are homotopic to identity (e.g., the cohomology calculations in Section \ref{SubSec:HodgeTheoryForDynamicalCohomology}, and both $(i)$ and $(ii)$ from Theorem \ref{MainThm:ThmB}). However, obtaining information about the homotopy type of a $\mathbb{Z}^{k}-$action from the assumption that $\Delta_{\alpha}$ is globally hypoelliptic is difficult. Indeed, even in the case of the torus (of high dimension), it is still not known if a cohomology free diffeomorphism is homotopic to identity.
\end{remark}

\section{Background and definitions}

For the remainder of this paper $M$ will be assumed to be a smooth closed manifold. Furthermore, any action on $M$ will be assumed to be smooth.

\subsection{Cohomology of dynamical systems}

Let $G$ be a connected Lie group with Lie algebra $\mathfrak{g}$. We begin by defining cohomology groups for representations of $\mathfrak{g}$ (the \textit{Chevalley-Eilenberg complex}) which will be used to define the cohomology of an action. Let $\pi:\mathfrak{g}\to\text{End}(V)$ be a representation on some space $V$. Let the space of $\ell-$\textit{cochains} be defined by:
\begin{align*}
C^{\ell}(\pi) = \text{Hom}\left(\Lambda_{\ell}(\mathfrak{g}),V\right)
\end{align*}
where $\Lambda_{\ell}(\mathfrak{g})$ is the $\ell$th exterior power of $\mathfrak{g}$. Define a differential map $\intd:C^{\ell}(\pi)\to C^{\ell+1}(\pi)$ by
\begin{align*}
(\intd\omega)(Y_{0},...,Y_{\ell}) = & \sum_{i = 0}^{\ell}(-1)^{i}\pi\left(Y_{i}\right)\left(\omega(Y_{0},...,\widehat{Y}_{i},...,Y_{\ell})\right) + \\ & +
\sum_{i < j}(-1)^{i+j}\omega([Y_{i},Y_{j}],...,\widehat{Y}_{i},...,\widehat{Y}_{j},...,Y_{\ell})
\end{align*}
where $\omega$ is an alternating multi-linear map and $\widehat{Y}_{i}$ indicates that $Y_{i}$ is omitted as an argument. A calculation shows $\intd^{2} = 0$. We define
\begin{align}
H^{\ell}(\pi) = \frac{\ker\left(\intd:C^{\ell}(\pi)\to C^{\ell+1}(\pi)\right)}{\text{Im}\left(\intd:C^{\ell-1}(\pi)\to C^{\ell}(\pi)\right)} =: \frac{Z^{\ell}(\pi)}{B^{\ell}(\pi)}.
\end{align}
We call the elements in $Z^{\ell}(\pi)$ \textit{cocycles} and the elements in $B^{\ell}(\pi)$ \textit{coboundaries}. For a locally free action $\alpha:G\times M\to M$ we obtain a map $\rho_{\alpha}:\mathfrak{g}\to\Gamma(\text{T}M)$ defined as
\begin{align*}
\rho_{\alpha}(X)(x) := \frac{\intd}{\intd t}\bigg|_{t = 0}\alpha(tX,x).
\end{align*}
For a vector $X\in\mathfrak{g}$ we will denote $\rho_{\alpha}(X)$ by $X$ (since $\alpha$ is locally free the map $\rho_{\alpha}$ is injective, so this is justified after identifying $\mathfrak{g}$ with its image in $\Gamma(TM)$).
\begin{definition}\label{def:Cohomology1}
Let $\alpha:G\times M\to M$ be a $G-$action. We define the representation of $\mathfrak{g}$ on $C^{\infty}(M)$ by letting each $X\in\mathfrak{g}$ act on $C^{\infty}(M)$ by differentiation. We denote the corresponding cochains, cocycles, coboundaries, and cohomology groups by $C^{\ell}(\alpha)$, $Z^{\ell}(\alpha)$, $B^{\ell}(\alpha)$ and $H^{\ell}(\alpha)$.
\end{definition}
It will be useful to have explicit formulas for the coboundary operators $\intd:C^{0}(\alpha)\to C^{1}(\alpha)$ and $\intd:C^{1}(\alpha)\to C^{2}(\alpha)$ in the special case $G = \mathbb{R}^{k}$. For $\ell = 0$ and $\ell = 1$ we obtain the maps:
\begin{align}
&\label{Eq:FirstCoboundaryOperator} \intd\omega(X) = X\omega, \\
&\label{Eq:SecondCoboundaryOperator} \intd\omega(X,Y) = X\omega(Y) - Y\omega(X).
\end{align}
Definition \ref{def:Cohomology1} can be generalized to cover many relevant cochains occuring in practice.
\begin{example}\label{Ex:Cohomology2}
Let $E\to M$ be a smooth vector bundle over $M$ and let $\mathcal{A}:G\times E\to E$ be a smooth linear cocycle in $E$ over the action $\alpha:G\times M\to M$. That is we have for $g\in G$, $x\in M$ a linear map
\begin{align*}
\mathcal{A}(g,x):E_{x}\to E_{\alpha(g)x}
\end{align*}
that satisfies the cocycle property
\begin{align*}
\mathcal{A}(g_{1}g_{2},x) = \mathcal{A}(g_{1},\alpha(g_{2})x)\mathcal{A}(g_{2},x).
\end{align*}
If $\Gamma(E)$ is the space of smooth sections of $E$, and $u\in\Gamma(E)$ then we define, for $X\in\mathfrak{g}$, an operator
\begin{align*}
\pi^{\mathcal{A}}(X)u(x) = \frac{\intd}{\intd t}\bigg|_{t = 0}\mathcal{A}(\exp(-tX),x)u(\alpha(\exp(tX))x)
\end{align*}
which makes $\pi^{\mathcal{A}}$ into a representation of $\mathfrak{g}$ on $\Gamma(E)$. This example includes the previous example by considering the trivial cocycle on the flat bundle $M\times\mathbb{C}\to M$. But this example also contains many other important cohomology groups, for example, we can choose $\mathcal{A}$ to be the derivative cocycle in $TM$, or the adjoint of this cocycle in $T^{*}M$. Given a finite dimensional representation $\beta:G\to{\rm GL}(V)$ we can also define a twisted cohomology in the bundle $M\times V\to M$ with the cocycle $\mathcal{A}(g,x) = \beta(g)$.
\end{example}
\begin{definition}\label{def:Cohomology2}
Let $\alpha:\mathbb{R}^{k}\times M\to M$ be a smooth action and $G$ a Lie group. We say that $\beta:\mathbb{R}^{k}\times M\to G$ is a smooth $G-$valued cocycle over $\alpha$ if $\beta$ is smooth and
\begin{align*}
\beta(\mathbf{a}_{1} + \mathbf{a}_{2},x) = \beta(\mathbf{a}_{1},\alpha(\mathbf{a}_{2})x)\beta(\mathbf{a}_{2},x)
\end{align*}
for all $\mathbf{a}_{1},\mathbf{a}_{2}\in\mathbb{R}^{k}$ and $x\in M$. We say that two $G-$valued cocycles $\beta,\beta'$ are cohomologous if there is $b:M\to G$ such that $\beta(\mathbf{a},x) = b(\alpha(\mathbf{a})x)^{-1}\beta'(\mathbf{a},x)b(x)$. We say that $\beta$ is constant if there is a homomorphism $\phi:\mathbb{R}^{k}\to G$ such that $\beta(\mathbf{a},x) = \phi(\mathbf{a})$, and we say that $\beta$ is trivial if $\beta(\mathbf{a},x) = e$. We denote by $Z(\alpha,G)$ the space of $G-$valued smooth cocycles over $\alpha$.
\end{definition}
Let $\beta$ be a $\mathbb{C}^{n}-$valued cocycle as in Definition \ref{def:Cohomology2}. Given $X\in\mathbb{R}^{k}$ we define $\omega(X)(x)$ as the derivative of $\beta(tX,x)$ at $t = 0$, then $\omega(aX+bY)(x) = a\omega(X)(x) + b\omega(Y)(x)$ since the derivative is linear. Moreover, the cocycle condition in Definition \ref{def:Cohomology2} implies:
\begin{align}
\beta(tX,\alpha(sY)x) = \beta(tX + sY,x) - \beta(sY,x) = \beta(sY,\alpha(tX)x) + \beta(tX,x) - \beta(sY,x)
\end{align}
so if we differentiate in $s$ and $t$ at $t,s = 0$ we obtain:
\begin{align}\label{Eq:CocycleConditionIntegratedCocycle}
Y(\omega(X))(x) = X(\omega(Y))(x).
\end{align}
If $(\omega_{1}(X),...,\omega_{n}(X)) = \omega(X)$ then equation \ref{Eq:CocycleConditionIntegratedCocycle} is equivalent to $\omega_{j}\in Z^{1}(\alpha)$ (see formula \ref{Eq:SecondCoboundaryOperator}). Denote the map $\beta\mapsto(\omega_{1},...,\omega_{n})$ by $T:Z(\alpha,\mathbb{C}^{n})\to Z^{1}(\alpha)^{n}$. We have the following important result, which follows by integration.
\begin{theorem}\label{Thm:Cohomology1}
The map $T:Z(\alpha,\mathbb{C}^{n})\to Z^{1}(\alpha)^{n}$ is bijective. Moreover, $T(\beta)\in B^{1}(\alpha)^{n}$ if and only if $\beta$ is cohomologous to the trivial cocycle. More generally, $\beta$ is cohomologous to a constant cocycle if and only if $T(\beta) = L + \omega_{B}$ where $\omega_{B}\in B^{1}(\alpha)^{n}$ and $L:\mathbb{R}^{k}\to\mathbb{C}^{n}$ is a linear map. 
\end{theorem}
\begin{proof}
We may assume, without loss of generality, that $n = 1$. Let $\omega\in Z^{1}(\alpha)$. Since the action $\alpha$ is locally free any $\alpha-$orbit is an immersed submanifold of $M$. That is, $\mathcal{O}_{\alpha}(x)$ is a manifold and using the action $\alpha$ we can identify $T_{y}\mathcal{O}_{\alpha}(x)$ with $\mathbb{R}^{k}$. We obtain local coordinates around $x$ in $\mathcal{O}_{\alpha}(x)$ by: $\mathbf{t} = (t^{1},...,t^{k})\mapsto\alpha(\mathbf{t})x$, and associated $1-$forms $\intd t^{j}$. For any $(X^{1},...,X^{k}) = X\in\mathbb{R}^{k}\cong T_{y}\mathcal{O}_{\alpha}(x)$ it is then clear that $\intd t^{j}(X) = X^{j}$. The cocycle $\omega$ naturally defines a $1-$form on the manifold $\mathcal{O}_{\alpha}(x)$ by $\omega(e_{1})\intd t^{1}+...+\omega(e_{k})\intd t^{k}$ where $e_{1},...,e_{k}\in\mathbb{R}^{k}$ is the standard basis. The de Rahm differential of $\omega$ viewed as a $1-$form on $\mathcal{O}_{\alpha}(x)$ can then be calculated as:
\begin{align*}
\intd(\omega(e_{1})\intd t^{1}+...+\omega(e_{k})\intd t^{k}) = & \sum_{j = 1}^{k}\left[\sum_{i = 1}^{k}e_{i}(\omega(e_{j}))\intd t^{i}\right]\wedge\intd t^{j} = \\ &
\sum_{i < j}\left[e_{i}(\omega(e_{j})) - e_{j}(\omega(e_{i}))\right]\intd t^{i}\wedge\intd t^{j} = 0
\end{align*}
where the last equality follows since $\omega\in Z^{1}(\alpha)$ (see formula \ref{Eq:SecondCoboundaryOperator}). That is, when we viewed as a $1-$form on $\mathcal{O}_{\alpha}(x)$, $\omega$ has de Rahm differential equal to $0$. It follows that for any piecewise smooth closed curve $\gamma$ in $\mathbb{R}^{k}$ we have
\begin{align}
0 = \int_{\alpha(\gamma)x}\omega.
\end{align}
As a consequence, the following function is well-defined
\begin{align}
\beta:\mathbb{R}^{k}\times M\to\mathbb{C},\quad\beta(\mathbf{t},x) = \int_{\alpha(\gamma)x}\omega,
\end{align}
where $\gamma$ is any smooth curve in $\mathbb{R}^{k}$ connecting $0$ and $\mathbf{t}$. Let $\mathbf{t},\mathbf{s}\in\mathbb{R}^{k}$ be two points, $\gamma_{1}:[0,1]\to\mathbb{R}^{k}$ be a smooth curve from $0$ to $\mathbf{s}$, $\gamma_{2}:[0,1]\to\mathbb{R}^{k}$ a smooth curve from $\mathbf{s}$ to $\mathbf{s} + \mathbf{t}$, and $\gamma_{2}\cdot\gamma_{1}$ the concatenation of the two curves. Note that $\gamma_{2}\cdot\gamma_{1}$ is a curve from $0$ to $\mathbf{s} + \mathbf{t}$ and $\gamma_{2} - \mathbf{s}$ is a curve from $0$ to $\mathbf{t}$. Using the definition of $\beta$ we obtain
\begin{align*}
\beta(\mathbf{t} + \mathbf{s},x) = & \int_{\alpha(\gamma_{2}\cdot\gamma_{1})x}\omega = \int_{\alpha(\gamma_{1})x}\omega + \int_{\alpha(\gamma_{2})x}\omega = \\ &
\int_{\alpha(\gamma_{1})x}\omega + \int_{\alpha(\gamma_{2} - \mathbf{s})\alpha(\mathbf{s})x}\omega = \beta(\mathbf{s},x) + \beta(\mathbf{t},\alpha(\mathbf{s})x),
\end{align*}
so $\beta$ is a $\mathbb{C}-$valued cocycle in the sense of Definition \ref{def:Cohomology2}. We define $S:Z^{1}(\alpha)\to Z(\alpha,\mathbb{C})$ by $S(\omega) = \beta$ where $\beta$ is constructed as above. We claim that $S$ is an inverse to $T$. Indeed, for $X\in\mathbb{R}^{k}$ and $t > 0$ we let $\gamma_{t}(s) = sX$, $s\in[0,t]$. We then have
\begin{align}
T(S(\omega))(X)(x) = \frac{\intd}{\intd t}\bigg|_{t = 0}\int_{\alpha(\gamma_{t})x}\omega = \frac{\intd}{\intd t}\bigg|_{t = 0}\int_{0}^{t}\omega(X)(\alpha(sX)x)\intd s = \omega(X)(x).
\end{align}
Conversely, if $\omega = T(\beta)$ then by using the line $\gamma_{1}(s) = sX$, for $X\in\mathbb{R}^{k}$ and $s\in[0,1]$, we obtain
\begin{align*}
S(\omega)(X,x) = & \int_{0}^{1}\omega(X)(\alpha(sX)x)\intd s = \int_{0}^{1}\frac{\intd}{\intd t}\bigg|_{t = 0}\beta(tX,\alpha(sX)x)\intd s = \\ &
\frac{\intd}{\intd t}\bigg|_{t = 0}\int_{0}^{1}\beta(tX,\alpha(sX)x)\intd s = \\ &
\frac{\intd}{\intd t}\bigg|_{t = 0}\int_{0}^{1}\left[\beta((t+s)X,x) - \beta(sX,x)\right]\intd s = \frac{\intd}{\intd t}\bigg|_{t = 0}\int_{0}^{1}\beta((t+s)X,x)\intd s = \\ &
\int_{0}^{1}\frac{\intd}{\intd s}\beta(sX,x)\intd s = \beta(X,x) - \beta(0,x) = \beta(X,x),
\end{align*}
where we used $\beta(0,x) = 0$, which follows since $\beta(0,x) = \beta(0+0,x) = \beta(0,\alpha(0)x) + \beta(0,x) = 2\beta(0,x)$.

If $b:M\to\mathbb{C}$ is smooth, then $\beta_{b}(\mathbf{t},x) = b(\alpha(\mathbf{t})x) - b(x)$ defines a $\mathbb{C}-$valued cocycle. We have
\begin{align}
T(\beta_{b})(X) = \frac{\intd}{\intd t}\bigg|_{t = 0}\left[b\circ\alpha(tX) - b\right] = Xb.
\end{align}
By formula \ref{Eq:FirstCoboundaryOperator} it follows that $T(\beta_{b})\in B^{1}(\alpha)$ for every $b\in C^{\infty}(M)$. Conversely, if $\omega\in B^{1}(\alpha)$ then there is some $b\in C^{\infty}(M)$ such that $\omega(X) = Xb$, so we obtain:
\begin{align}
S(\omega)(X,x) = \int_{0}^{1}\omega(X)(\alpha(sX)x)\intd s = \int_{0}^{1}Xb(\alpha(sX)x)\intd s = b(\alpha(X)x) - b(x)
\end{align}
so $S(\omega)$ is a coboundary in the sense of Definition \ref{def:Cohomology2}. Since any $\mathbb{C}-$valued cocycle cohomologous to $0$ is given by $\beta_{b}$ for some $b$ the second to last claim of the theorem follows. If $\beta$ is cohomologous to a constant cocycle, then there is a linear map $L:\mathbb{R}^{k}\to\mathbb{C}$ and a $b\in C^{\infty}(M)$ such that $\beta = \beta_{b} + L$. Since $T(\beta_{b})\in B^{1}(\alpha)$ it suffices to show that $T(L):\mathbb{R}^{k}\to\mathbb{C}$ is a linear map. This is immediate since
\begin{align}
T(L)(X) = \frac{\intd}{\intd t}\bigg|_{t = 0}L(tX) = L(X).
\end{align}
Conversely, if $\omega\in Z^{1}(\alpha)$ can be written as $\omega = \omega_{B} + L$ with $\omega_{B}\in B^{1}(\alpha)$ and $L$ linear, then $S(\omega_{B}) = \beta_{b}$ for some $b\in C^{\infty}(M)$, so it suffices to show that $S(L)$ is a homomorphism. To see this, note that we already showed that $T(L) = L$, so $S(L) = S(T(L)) = L$ since $S$ is the inverse of $T$.
\end{proof}

\subsection{Globally hypoelliptic operators and tame estimates}

We denote by $\norm{\cdot}_{n}$ the $C^{n}-$norm or $n$th Sobolev norm on $C^{\infty}(M)$ (by the Sobolev embedding theorem the choice between Sobolev norms or uniform norms makes no difference). We will always consider $C^{\infty}(M)$ as the space of smooth $\mathbb{C}-$valued functions.

Let $\mathcal{D}'(M) := (C^{\infty}(M))'$ be the space of distributions on $M$. That is, $\mathcal{D}'(M)$ consists of continuous linear functionals on $C^{\infty}(M)$ where $C^{\infty}(M)$ is endowed with its graded Fréchet structure given by $\norm{\cdot}_{n}$ (see \cite{Hamilton}). Let $\Omega^{\ell}(M)$ denote the space of $\ell-$forms on $M$. In particular, $\Omega^{d}(M)$ is the space of top forms. Any top form $\omega\in\Omega^{d}(M)$ naturally defines an element in $\mathcal{D}'(M)$ by integration:
\begin{align}\label{Eq:BackgroundTameEstimates1}
\omega(f) = \int_{M}f\omega,\quad f\in C^{\infty}(M).
\end{align}
Using this identification $\Omega^{d}(M)\to\mathcal{D}'(M)$ we consider $\Omega^{d}(M)$ as a subspace of $\mathcal{D}'(M)$.

Given a continuous linear operator $L:C^{\infty}(M)\to C^{\infty}(M)$ we consider the induced operator $L':\mathcal{D}'(M)\to\mathcal{D}'(M)$ by $L'D(f) = D(Lf)$.
\begin{definition}\label{Def:BackgroundTameEstimates1}
We say that an operator $L:C^{\infty}(M)\to C^{\infty}(M)$ is globally hypoelliptic ({\rm GH}) if any distributional solution $L'D = \omega$ with $\omega\in\Omega^{d}(M)$ satisfy $D\in\Omega^{d}(M)$.
\end{definition}
Let $\mu$ be a volume form on $M$. We will consider operators $L:C^{\infty}(M)\to C^{\infty}(M)$ with a well-defined adjoint $L^{*}:C^{\infty}(M)\to C^{\infty}(M)$. We say that $L^{*}:C^{\infty}(M)\to C^{\infty}(M)$ is an adjoint of $L$ if $L'(\overline{f}\mu) = \overline{L^{*}f}\cdot\mu$, where $\overline{f}$ is the complex conjugate to $f$ (note that the adjoint depends on $\mu$). It is immediate that if $L^{*}$ is an adjoint of $L$, then $L$ is an adjoint of $L^{*}$. That is, we have $(L^{*})^{*} = L$. With this definition of $L^{*}$ we immediately obtain
\begin{align*}
\langle Lg,f\rangle = \langle g,L^{*}f\rangle,\quad\langle u,v\rangle = \int_{M}u(x)\overline{v(x)}\intd\mu(x),
\end{align*}
which shows, in particular, that $L^{*}$ is unique since it is the $L^{2}-$adjoint of $L$. Define $L:\mathcal{D}'(M)\to\mathcal{D}'(M)$ by $LD(f) = D(L^{*}f)$. We embed $C^{\infty}(M)\to\mathcal{D}'(M)$ using $f\mapsto\overline{f}\mu$.
\begin{lemma}\label{L:BackgroundTameEstimates1}
Let $L$ be an operator $L:C^{\infty}(M)\to C^{\infty}(M)$ with adjoint $L^{*}:C^{\infty}(M)\to C^{\infty}(M)$. Then $L^{*}$ is {\rm GH} if and only if for any distributional solution $LD = u$ with $u\in C^{\infty}(M)$ we have $D\in C^{\infty}(M)$. Moreover $L$ is {\rm GH} if and only if for every distributional $L^{*}D = u$ with $u\in C^{\infty}(M)$ we have $D\in C^{\infty}(M)$.
\end{lemma}
\begin{proof}
Suppose that $L^{*}$ is GH. Let $D\in\mathcal{D}'(M)$ and $f\in C^{\infty}(M)$ be such that $LD = f$. It follows that
\begin{align*}
\int_{M}g(x)\overline{f}(x)\intd\mu(x) = D(L^{*}g)
\end{align*}
or $(L^{*})'D = \overline{f}\mu\in\Omega^{d}(M)$ which implies, since $L^{*}$ is GH, that $D\in\Omega^{d}(M)$. Noting that any $\omega\in\Omega^{d}(M)$ can be written $\omega = \overline{f}_{\omega}\mu$ with $f_{\omega}\in C^{\infty}(M)$ the first direction of the lemma follows. Conversely, suppose that any distributional solution $LD = u\in C^{\infty}(M)$ is smooth. Note that $LD = (L^{*})'D$. Let $\omega\in\Omega^{d}(M)$, we write $\omega = \overline{f}_{\omega}\mu$ for some $f_{\omega}\in C^{\infty}(M)$. If $(L^{*})'D = \omega$ then $LD = f_{\omega}$, so $D\in C^{\infty}(M)$. That is, we can write $D = \overline{g}\mu$ for some $g\in C^{\infty}(M)$ which proves the lemma since $\overline{g}\mu\in\Omega^{d}(M)$.

The last claim of the lemma follows since $(L^{*})^{*} = L$.
\end{proof}
We have the following important properties of GH operators, proved in Appendix \ref{Appendix:A}.
\begin{lemma}\label{L:BackgroundTameEstimates2}
Let $\mu$ be a volume form on $M$ and let $L:C^{\infty}(M)\to C^{\infty}(M)$ be an operator with adjoint $L^{*}:C^{\infty}(M)\to C^{\infty}(M)$. If $L^{*}$ is {\rm GH} then $\dim\ker L < \infty$ and $L$ has closed image. If $L$ is also {\rm GH} then $\dim{\rm coker}(L) = \dim\ker L^{*} < \infty$. In particular, if $L$ is self-adjoint, $L^{*} = L$, and $L$ is {\rm GH} then $\dim\ker L = \dim{\rm coker}L < \infty$.
\end{lemma}
\begin{remark}
This theorem is proved in \cite{ForniGLWdim3} with $L$ given by the differential operator induced by a vector field, and in \cite{ChenChiEqGW} for vector fields on tori. The proof we give is essentially the same but for a more general class of operators that includes, in particular, the orbitwise laplacian of a smooth action, see Definition \ref{def:OrbitwiseLaplacian}.
\end{remark}
\begin{definition}\label{Def:BackgroundTameEstimates2}
We say that a linear operator $L:F\to F$ on a graded Fréchet space $F$ (see \cite{Hamilton}) satisfy a tame estimate if there are integers $b,r_{0}\in\mathbb{N}_{0}$ and constants $C_{r}$ such that for $r\geq b$ we have
\begin{align*}
\norm{Lv}_{r}\leq C_{r}\norm{v}_{r + r_{0}}.
\end{align*}
\end{definition}

\subsection{The orbitwise laplacian}

Let $G$ be a Lie group with Lie algebra $\mathfrak{g}$. We fix an inner product $\langle\cdot,\cdot\rangle$ on $\mathfrak{g}$ (when $G = \mathbb{R}^{k}$ we will always choose the inner product on $\mathbb{R}^{k}$ to be the standard inner product). Let $\alpha:G\times M\to M$ be locally free action on a closed smooth manifold $M$.
\begin{definition}\label{def:OrbitwiseLaplacian}
We define the orbitwise laplacian of $\alpha$ by
\begin{align}\label{Eq:OrbitwiseLaplacian1}
\Delta_{\alpha}u = -\sum_{j = 1}^{k}X_{j}^{2}u,\quad u\in C^{\infty}(M)
\end{align}
where $X_{1},...,X_{k}\in\mathfrak{g}$ is an orthonormal basis.
\end{definition}
\begin{remark}
It is standard that \ref{Eq:OrbitwiseLaplacian1} does not depend on the choice of ON-basis. If $G = \mathbb{R}^{k}$ we always choose $X_{j}$ as the $j$th unit vector.
\end{remark}
\begin{remark}\label{Rmk:SelfAdjointness}
It is immediate that if $\alpha$ preserves a volume form $\mu$, then $\Delta_{\alpha}$ is self-adjoint. That is, $\Delta_{\alpha}$ is an adjoint to $\Delta_{\alpha}$.
\end{remark}
\begin{remark}
The orbitwise laplacian coincides with the leafwise laplacian of the orbit foliation of $\alpha$, if we give the orbit foliation a $G-$invariant metric. We prefer to use separate terminology since one can also define an orbitwise laplacian for a $\mathbb{Z}^{k}-$action, but in this case there is no leafwise laplacian.
\end{remark}
\begin{definition}\label{Def:OrbitwiseLaplacian2}
Let $\alpha$ be a locally free, smooth $G-$action. We say that $\alpha$ is Globally Hypoelliptic $({\rm GH})$ if the orbitwise laplacian $\Delta_{\alpha}$ is globally hypoelliptic.
\end{definition}
Abelian actions with globally hypoelliptic orbitwise laplacian (or simply ${\rm GH}$ actions) will be the main object of study in this paper.

\subsection{Nilpotent Lie groups}

Let $G$ be a connected, simply connected Lie group with Lie algebra $\mathfrak{g}$. We define the \textit{lower central series}, $\mathfrak{g}_{(j)}$, of $\mathfrak{g}$ inductively by
\begin{align*}
\mathfrak{g}_{(1)} = \mathfrak{g},\quad\mathfrak{g}_{(j+1)} = [\mathfrak{g},\mathfrak{g}_{(j)}].
\end{align*}
If $\mathfrak{g}_{(N)} = 0$ for some $N$ then we say that $\mathfrak{g}$, and $G$, are nilpotent. If $\ell$ is the maximal integer such that $\mathfrak{g}_{(\ell)}\neq0$ then $\ell$ is the step of $\mathfrak{g}$ and $G$.

On any Lie group $G$ there exists a unique (up to scaling) left-invariant volume form $\mu$ \cite{HarmonicAnalysis}, the Haar measure of $G$. Let $\Gamma\leq G$ be a discrete subgroup. We say that $\Gamma$ is a lattice if there is a right-translation-invariant measure $\mu_{\Gamma}$ on $\Gamma\setminus G$ such that $\mu_{\Gamma}(\Gamma\setminus G) < \infty$ (if such a measure exists, then we will assume that it is normalized $\mu_{\Gamma}(\Gamma\setminus G) = 1$, this normalization makes the measure unique). If $\Gamma\leq G$ is a lattice and $\Gamma\setminus G$ is compact, then $\Gamma$ is a uniform lattice. In the case of nilpotent Lie groups every lattice is uniform \cite{CorwinGreenleaf}.

Let $G$ be a simply connected nilpotent Lie group. Given a lattice $\Gamma\leq G$ we will denote the associated \textit{compact nilmanifold} by 
\begin{align*}
M_{\Gamma} := \Gamma\setminus G.
\end{align*}
If $\mathfrak{g}_{\mathbb{Q}}$ is a rational Lie algebra such that $\mathfrak{g}_{\mathbb{Q}}\otimes\mathbb{R} = \mathfrak{g}$ then we say that $\mathfrak{g}_{\mathbb{Q}}$ is a rational structure in $\mathfrak{g}$. There is a $1-1$ correspondence between rational structures of $\mathfrak{g}$ and commensurability classes of lattices in $G$ \cite{CorwinGreenleaf}. In particular, if we fix a lattice $\Gamma$ of $G$ then there is an associated rational structure on $\mathfrak{g}$\footnote{which can be defined as ${\rm span}_{\mathbb{Q}}(\log\Gamma)$ where $\log:G\to\mathfrak{g}$ is the inverse of the exponential map}. When a lattice is fixed, we will always consider the rational structure induced by the lattice as $\mathfrak{g}_{\mathbb{Q}}$.

We will be especially interested in two examples of nilpotent Lie groups.
\begin{example}[Heisenberg group]\label{Ex:HeisenberGroup}
We define 
\begin{align*}
\mathfrak{h}^{g} := \text{span}(X_{1},...,X_{g},Y_{1},...,Y_{g},Z)
\end{align*}
with brackets $[X_{i},Y_{i}] = Z$. Then $\mathfrak{h}^{g}$ is a $2-$step nilpotent Lie algebra. Let $H^{g}$ be the associated simply connected Lie group. We call $H^{g}$ the $g$th Heisenberg group (or the $(2g + 1)-$dimensional Heisenberg group).
\end{example}
\begin{example}[Quasi-abelian nilpotent Lie groups]\label{Ex:ModelFiliform}
Following the notation in \cite{FlaminioForni2023} we say that a simply connected nilpotent Lie group $G$ is \textit{quasi-abelian} if $G$ has a codimension $1$ normal abelian subgroup. The quasi-abelian nilpotent Lie groups can be constructed explicitly on the Lie algebra level as follows. Let $n_{1},...,n_{\ell}\in\mathbb{N}$ be integers, $A_{j}:\mathbb{R}^{n_{j}}\to\mathbb{R}^{n_{j+1}}$, $j < \ell$, be surjective maps, and let $A_{\ell}:\mathbb{R}^{n_{\ell}}\to0$ be the zero map. We define:
\begin{align}
\Hat{\mathfrak{g}} = \mathbb{R}\oplus\mathbb{R}^{n_{1}}\oplus...\oplus\mathbb{R}^{n_{\ell}}
\end{align}
and let $X$ generate the first direction. For $Y_{j}\in\mathbb{R}^{n_{j}}$ we define brackets in $\Hat{\mathfrak{g}}$ by:
\begin{align}
[X,Y_{j}] = A_{j}Y_{j}\in\mathbb{R}^{n_{j+1}}.
\end{align}
Then the Lie algebra of any quasi-abelian nilpotent Lie group can be written as $\mathfrak{g} = \Hat{\mathfrak{g}}\oplus\mathbb{R}^{N}$ for some integer $N$. Moreover, if $\mathfrak{g}$ is quasi-abelian with some rational structure $\mathfrak{g}_{\mathbb{Q}}$ then we can always make the identification as above such that the maps $A_{j}$ are rational (Lemma \ref{L:RationalityOfQuasiAbelianGroups}).
\end{example}

\section{Properties of the orbitwise laplacian}

In the remainder of this section, let $M$ be a closed, smooth $d-$dimensional manifold and $\alpha:G\to{\rm Diff}^{\infty}(M)$ a smooth action. The following lemma will be used in the remainder.
\begin{lemma}\label{L:InvarianceImpliesKernelOfLaplacian}
The orbitwise laplacian is formally self-adjoint with respect to any $\alpha-$invariant measure. If $\alpha$ preserve a volume form, then a function $u\in C^{\infty}(M)$ is $\alpha-$invariant if and only if $\Delta_{\alpha}u = 0$.
\end{lemma}
\begin{proof}
The first claim was pointed out in Remark \ref{Rmk:SelfAdjointness}. Let $u,v\in C^{\infty}(M)$ and let $\mu$ be any $\alpha-$invariant ergodic measure (by the ergodic decomposition it suffices to consider ergodic measures). Let $X\in\mathfrak{g}$ be a generator of $\alpha$, then $\alpha-$invariance of $\mu$ and the product rule for differentiation shows
\begin{align}
\int_{M}(Xu)v\intd\mu = -\int_{M}u(Xv)\intd\mu,\quad\int_{M}(X^{2}u)v\intd\mu = -\int_{M}(Xu)(Xv)\intd\mu = \int_{M}u(X^{2}v)\intd\mu,
\end{align}
which implies that $\Delta_{\alpha}$ is formally self-adjoint. If $u\in C^{\infty}(M)$ is $\alpha-$invariant then clearly $\Delta_{\alpha}u = 0$. Conversely, if $\Delta_{\alpha}u = 0$ and $\mu$ is an $\alpha-$invariant volume form, then
\begin{align}
0 = \int_{M}u\Delta_{\alpha}\overline{u}\intd\mu = \sum_{j = 1}^{k}\int_{M}X_{j}u\cdot X_{j}\overline{u}\intd\mu = \sum_{j = 1}^{k}\int_{M}|X_{j}u|^{2}\intd\mu
\end{align}
so $X_{j}u = 0$ almost surely with respect to $\mu$. Since $\mu$ is supported everywhere, it follows that $X_{j}u = 0$ so $u$ is $\alpha-$invariant.
\end{proof}
The following Lemma is contained in \cite{DanijelaGH1,DanijelaGH2}, we include a proof for completeness.
\begin{theorem}\label{Thm:PropertiesOrbitLaplacian1}
Let $G$ be a connected Lie group and let $\alpha:G\times M\to M$ a ${\rm GH}-$action. If $\alpha$ preserves an ergodic volume form $\mu$, then the space of $\alpha-$invariant distributions is spanned by $\mu$ (and is in particular $1-$dimensional). If $G\cong\mathbb{R}^{k}$, then $\alpha$ preserve an ergodic volume.
\end{theorem}
\begin{remark}\label{Rmk:MinimalityRemark}
When $G$ is abelian Theorem \ref{Thm:PropertiesOrbitLaplacian1} shows that $\alpha$ is, in particular, uniquely ergodic (and minimal since the unique invariant measure is everywhere supported). In the remainder, $\mu$ will always denote the unique $\alpha-$invariant measure when $\alpha$ is GH and $G\cong\mathbb{R}^{k}$.
\end{remark}
\begin{proof}
Let $D$ be a $\alpha-$invariant distribution, then $\Delta_{\alpha}D = 0$. Since $\Delta_{\alpha}$ is GH $D$ is represented by some smooth top form $\omega\in\Omega^{d}(M)$. If $\alpha$ preserve a volume form $\mu$, then we can write $\omega = \varphi\mu$ and
\begin{align}
0 = \int_{M}(\Delta_{\alpha}u)(x)\varphi(x)\intd\mu(x) = \int_{M}u(x)(\Delta_{\alpha}\varphi)(x)\intd\mu(x),\quad u\in C^{\infty}(M)
\end{align}
which implies that $\Delta_{\alpha}\varphi = 0$. By Lemma \ref{L:InvarianceImpliesKernelOfLaplacian} this implies that $\varphi$ is $\alpha-$invariant, which implies that $\varphi$ is constant by ergodicity (and the fact that $\mu$ is supported everywhere).

If $G\cong\mathbb{R}^{k}$ then there is at least on $\alpha-$invariant measure. Any $\alpha-$invariant measure, $\nu$, defines a $\alpha-$invariant distribution by integration. It follows that any $\alpha-$invariant measure $\nu$ is represented by some top form $\omega_{\nu}$. Since the support of any top form contain an open set there are at most countably many $\alpha$-invariant ergodic measures. Define the open set $U_{\nu} = \{\omega_{\nu}\neq0\}$ for each $\alpha-$invariant ergodic measure $\nu$, then $U_{\nu_{1}}\cap U_{\nu_{2}} = \emptyset$ for $\nu_{1}\neq\nu_{2}$. If there are at least two $\alpha-$invariant ergodic measures then the set
\begin{align}
C = \bigcap_{\nu}U_{\nu}^{c},\quad\nu\text{ is ergodic}
\end{align}
is compact, non-empty, and $\alpha-$invariant. But then $C$ supports an $\alpha$-invariant ergodic measure, which is a contradiction. We conclude that $\alpha$ has a unique invariant measure $\mu$. The set $U_{\mu}^{c}$ is $\alpha-$invariant and compact, so if it is non-empty it support an $\alpha-$invariant measure. This would contradict unique ergodicity of $\alpha$, so $U_{\mu}^{c} = \emptyset$ or $U_{\mu} = M$. We conclude that $\omega_{\mu}$ is a volume form.
\end{proof}
\begin{theorem}\label{Thm:PropertiesOrbitLaplacian2}
Let $G$ be a connected Lie group and $\alpha:G\times M\to M$ an action with invariant ergodic volume. Then the action $\alpha$ is ${\rm GH}$ if and only if ${\rm Im}(\Delta_{\alpha})$ has codimension $1$. More precisely we have 
\begin{align*}
{\rm Im}(\Delta_{\alpha}) = C_{0}^{\infty}(M) = \left\{f\in C^{\infty}(M)\text{ : }\int_{M}f\intd\mu = 0\right\}
\end{align*}
where $\mu$ is the unique $\alpha-$invariant measure.
\end{theorem}
\begin{proof}
Let $\alpha$ be GH. By Theorem \ref{Thm:PropertiesOrbitLaplacian1} the space of $\alpha-$invariant distributions is $1-$dimensional. This implies, in particular, that the kernel $\ker\Delta_{\alpha}$ has dimension $1$ (since every $u\in\ker\Delta_{\alpha}$ defines a $\alpha-$invariant distribution by $\overline{u}\mu\in\mathcal{D}'(M)$). By Lemma \ref{L:BackgroundTameEstimates2} it follows that 
\begin{align}
\dim(C^{\infty}(M)/{\rm Im}(\Delta_{\alpha})) = 1.
\end{align}
Note that $\Delta_{\alpha}u\in\ker\mu$ for every $u\in C^{\infty}(M)$, so since 
\begin{align}
\dim(C^{\infty}(M)/{\rm Im}(\Delta_{\alpha})) = \dim(C^{\infty}(M)/\ker\mu) = 1
\end{align}
we have ${\rm Im}(\Delta_{\alpha}) = \ker\mu$. To show the converse, suppose that $D\in\mathcal{D}'(M)$ satisfy $\Delta_{\alpha}D = g\in C^{\infty}(M)$. Then, we find $f\in C^{\infty}(M)$ such that $\Delta_{\alpha}f = g$ (note that $g$ has zero integral since the integral of $g$ coincide with its pairing to $1$). So, $\Delta_{\alpha}(D - f) = 0$. For $v\in C_{0}^{\infty}(M)$ we find $u\in C_{0}^{\infty}(M)$ so that $\Delta_{\alpha}u = v$, so $(D-f)(v) = (D-f)(\Delta_{\alpha}u) = \Delta_{\alpha}(D - f)u = 0$. It follows that $D = f$ on $C_{0}^{\infty}(M)$, which implies $D = f + D(1)\in C^{\infty}(M)$ so $\alpha$ is GH by Lemma \ref{L:BackgroundTameEstimates1}.
\end{proof}
\begin{lemma}\label{L:PropertiesOrbitLaplacian2}
Let $\alpha:G\times M\to M$ be ${\rm GH}$ with an invariant ergodic volume $\mu$, $\beta:G\times B\to B$ a $G-$action and $\pi:M\to B$ a submersion such that $\pi\circ\alpha(g) = \beta(g)\circ\pi$. Then $\beta$ is ${\rm GH}$ with a unique invariant ergodic volume $\nu = \pi_{*}\mu$.
\end{lemma}
\begin{proof}
By Theorem \ref{Thm:PropertiesOrbitLaplacian1} the volume $\mu$ is the unique $\alpha-$invariant measure. Write $\nu = \pi_{*}\mu$. It is immediate that $\nu$ is $\beta-$invariant and ergodic. By disintegrating $\mu$ over $\pi$ we obtain a decomposition
\begin{align*}
\mu = \int_{B}\mu_{b}^{\pi}\intd\nu(b)
\end{align*}
where $\mu_{b}^{\pi}$ is a probability measure on $\pi^{-1}(b)$. Define $\pi_{*}:C^{\infty}(M)\to C^{\infty}(B)$ by fiber integration
\begin{align}
\pi_{*}(f)(b) = \int_{\pi^{-1}(b)}f(x)\intd\mu_{b}^{\pi}(x)
\end{align}
and $\pi^{*}$ by pull-back $\pi^{*}f = f\circ\pi$. The disintegration $(\mu_{b}^{\pi})_{b\in B}$ is unique, which implies $\alpha(g)_{*}\mu_{b}^{\pi} = \mu_{\beta(g)b}^{\pi}$. It follows that $\pi_{*}$ intertwines the action of $\alpha$ and $\beta$:
\begin{align*}
\pi_{*}(f)\circ\beta(g)(b) = & \int_{M}f(x)\intd\mu_{\beta(g)b}^{\pi}(x) = \int_{M}f(x)\intd\alpha(g)_{*}\mu_{b}^{\pi}(x) = \\ = &
\int_{M}f(\alpha(g)x)\intd\mu_{b}^{\pi}(x) = \pi_{*}(f\circ\alpha(g)).
\end{align*}
So, for $X\in\mathfrak{g}$ we have $X\pi_{*}(f) = \pi_{*}(Xf)$ and $\Delta_{\beta}\pi_{*}(f) = \pi_{*}(\Delta_{\alpha}f)$. Since $\nu$ is $\beta-$invariant it suffices to show (Lemma \ref{L:BackgroundTameEstimates1}) that if $\Delta_{\beta}D = g\in C^{\infty}(B)$ then $D = u\in C^{\infty}(B)$. Suppose that $\Delta_{\beta}D = g\in C^{\infty}(B)$. We define $\Tilde{D} = D\circ\pi_{*}\in\mathcal{D}'(M)$ and calculate $\Delta_{\alpha}\Tilde{D}$:
\begin{align*}
\Delta_{\alpha}\Tilde{D}(f) = & \Tilde{D}(\Delta_{\alpha}f) = D(\pi_{*}\Delta_{\alpha}f) = D(\Delta_{\beta}\pi_{*}f) = \Delta_{\beta}D(\pi_{*}f) = \\ = &
\int_{B}\left(\int_{M}f(x)\intd\mu_{b}^{\pi}(x)\right)\overline{g(b)}\intd\nu(b) = \\ = & 
\int_{B}\left(\int_{M}f(x)\overline{g(\pi(x))}\intd\mu_{b}^{\pi}(x)\right)\intd\nu(b) = \int_{M}f(x)\overline{\pi^{*}g}\intd\mu(x).
\end{align*}
So, $\Delta_{\alpha}\Tilde{D} = \pi^{*}g$. Since $\Delta_{\alpha}$ is GH it follows that $\Tilde{D} = u\in C^{\infty}(M)$. Finally,
\begin{align*}
D(f) = & D(\pi_{*}\pi^{*}f) = \Tilde{D}(\pi^{*}f) = \int_{M}f(\pi(x))\overline{u(x)}\intd\mu(x) = \\ = &
\int_{B}\left(\int_{M}f(b)\overline{u(x)}\intd\mu_{b}^{\pi}(x)\right)\intd\nu(b) = \int_{B}f(b)\overline{\pi_{*}u(b)}\intd\nu(b)
\end{align*}
so $D = \pi_{*}u\in C^{\infty}(B)$, which imlpies that $\beta$ is GH.
\end{proof}

\subsection{Orbitwise laplacian and cohomology}\label{SubSec:HodgeTheoryForDynamicalCohomology}

In the remainder of this section we are exclusively interested in abelian actions $\alpha:\mathbb{R}^{k}\times M\to M$. We fix such an action and recall that if $X\in\mathbb{R}^{k}$ then we identify $X$ with the associated vector field on $M$ that generates the flow $\alpha(tX)$, and if $u\in C^{\infty}(M)$ then $Xu$ is the derivative of $u$ along $X$. Let $e_{1},...,e_{k}\in\mathbb{R}^{k}$ be the standard basis and let $e^{1},...,e^{k}$ be the associated dual basis (defined by $e^{i}(e_{j}) = 1$ for $i = j$ and $0$ otherwise). Let $\mathcal{P}_{\ell}$ be defined as all $I = (i_{1},...,i_{\ell})$, $1\leq i_{1} < i_{2} < ... < i_{\ell}\leq k$. For $I\in\mathcal{P}_{\ell}$, denote $e^{I} = e^{i_{1}}\wedge...\wedge e^{i_{\ell}}\in\Lambda^{\ell}(\mathbb{R}^{k})$ and $e_{I} = e_{i_{1}}\wedge...\wedge e_{i_{\ell}}\in\Lambda_{\ell}(\mathbb{R}^{k})$ (where $\Lambda_{\ell}(\mathbb{R}^{k})$ is the  $\ell$th exterior power of $\mathbb{R}^{k}$ and $\Lambda^{\ell}(\mathbb{R}^{k})$ is the $\ell$th exterior power of $(\mathbb{R}^{k})^{'}$). Let $\langle\cdot,\cdot\rangle_{\ell,0}$ be the inner product on $\Lambda^{\ell}(\mathbb{R}^{k})$ such that $(e^{I})_{I}$ form an ON-basis. In this section, we will use the identification
\begin{align}
C^{\ell}(\alpha) = \text{Hom}(\Lambda_{\ell}(\mathbb{R}^{k}),C^{\infty}(M))\cong\Lambda^{\ell}(\mathbb{R}^{k})\otimes C^{\infty}(M) = C^{\infty}(M,\Lambda^{\ell}(\mathbb{R}^{k}))
\end{align}
which is obtained by noting that if $\omega\in\Lambda^{\ell}(\mathbb{R}^{k})$ and $u\in C^{\infty}(M)$ then we can define an element of $C^{\ell}(\alpha)$ by $X_{1}\wedge...\wedge X_{\ell}\mapsto\omega(X_{1}\wedge...\wedge X_{\ell})\cdot u$ and this map is a bijection. Let $\langle\cdot,\cdot\rangle_{\ell}$ be an inner product on $C^{\ell}(\alpha)$ defined by
\begin{align}
\langle\omega,\eta\rangle_{\ell} = \int_{M}\langle\omega(x),\eta(x)\rangle_{\ell,0}\intd\mu(x)
\end{align}
for $\mu$ some $\alpha-$invariant measure, and we have used $C^{\ell}(\alpha)\cong C^{\infty}(M,\Lambda^{\ell}(\mathbb{R}^{k}))$ (this inner product depends on the choice of invariant measure, but we will only use it to define an adjoint for the differential $\intd^{\ell}:C^{\ell}(\alpha)\to C^{\ell+1}(\alpha)$ and all $\alpha-$invariant measures defines the same adjoint, see Lemma \ref{L:FormulaForChainLaplacian}). For $\ell$ we have $\intd^{\ell}:C^{\ell}(\alpha)\to C^{\ell+1}(\alpha)$, let $\intd^{\ell,*}:C^{\ell+1}(\alpha)\to C^{\ell}(\alpha)$ be its formal adjoint. We define a laplacian on each cochain space $C^{\ell}(\alpha)$ as
\begin{align}
\Delta_{\alpha,\ell} := \intd^{\ell,*}\intd^{\ell} + \intd^{\ell-1}\intd^{\ell-1,*}.
\end{align}
It is immediate that $\Delta_{\alpha,0} = \Delta_{\alpha}$, with $\Delta_{\alpha}$ the orbitwise laplacian (Definition \ref{def:OrbitwiseLaplacian} and formula \ref{Eq:FirstCoboundaryOperator}).

We begin by deriving a formula for $\intd:C^{\ell}(\alpha)\to C^{\ell+1}(\alpha)$ under the identification $C^{\ell}(\alpha)\cong C^{\infty}(M)\otimes\Lambda^{\ell}(\mathbb{R}^{k})$.
\begin{lemma}\label{L:FormulaForDifferential}
For $I\in\mathcal{P}_{\ell}$ and $u\in C^{\infty}(M)$ we have
\begin{align*}
\intd(ue^{I}) = \sum_{j = 1}^{k}e_{j}u\cdot e^{j}\wedge e^{I}.
\end{align*}
\end{lemma}
\begin{proof}
The lemma follows from a calculation. Let $I = (i_{1},...,i_{\ell})\in\mathcal{P}_{\ell}$ and $J = (j_{0},j_{1},...,j_{\ell})\in\mathcal{P}_{\ell+1}$. It is immediate
\begin{align*}
\intd(ue^{I})(e_{J}) = & \sum_{r = 0}^{\ell}(-1)^{r}e_{j_{r}}u\cdot e^{I}(e_{J\setminus\{j_{r}\}}) = \\ = &
\sum_{r = 0}^{\ell}(-1)^{r}e_{j_{r}}u\cdot\begin{cases}
0,\quad I\not\subset J \\
0,\quad I\subset J\text{ and }j_{r}\in I \\
1,\quad I\subset J\text{ and }j_{r}\not\in I.
\end{cases}
\end{align*}
So
\begin{align*}
\intd(ue^{I}) = \sum_{r\not\in I}e_{r}u\cdot e^{r}\wedge e^{I} = \sum_{j = 1}^{k}e_{j}u\cdot e^{j}\wedge e^{I}
\end{align*}
where the signs cancel out from moving $e^{r}$ from its position within $e^{I\cup\{r\}}$ to the far left.
\end{proof}
Before proceeding, we define two operators. Let $E_{j}:\Lambda^{\ell}(\mathbb{R}^{k})\to\Lambda^{\ell+1}(\mathbb{R}^{k})$ and $\iota_{j}:\Lambda^{\ell}(\mathbb{R}^{k})\to\Lambda^{\ell-1}(\mathbb{R}^{k})$ be defined by
\begin{align}
& E_{j}(\omega) := e^{j}\wedge\omega, \\
& \iota_{j}(\omega)(X_{1},...,X_{\ell-1}) := \omega(e_{j},X_{1},...,X_{\ell-1}),\quad X_{1},...,X_{\ell-1}\in\mathbb{R}^{k}.
\end{align}
We also extend $E_{j}$ and $\iota_{j}$ to $C^{\ell}(\alpha)$ in the obvious way. If $I\in\mathcal{P}_{\ell}$ and $J\in\mathcal{P}_{\ell+1}$ then
\begin{align}\label{Eq:AdjointOfContraction}
\langle E_{j}(e^{I}),e^{J}\rangle_{\ell+1,0} = \langle e^{I},\iota_{j}(e^{J})\rangle_{\ell+1,0},
\end{align}
since both sides of the equality are non-zero precisely when $J = I\cup\{j\}$, and if $J = I\cup\{j\}$ then $e^{j}\wedge e^{I} = (-1)^{r}e^{J}$ for some $r$ determined by the position of $j$ in $J$, but we also have $\iota_{j}(e^{J})(e_{I}) = e^{J}(e_{j}\wedge e_{I}) = (-1)^{r}e^{J}(e_{J}) = (-1)^{r}$ so $\iota_{j}(e^{J}) = (-1)^{r}e^{I}$.
\begin{lemma}\label{L:FormulaForChainLaplacian}
For every $I = (i_{1},i_{2},...,i_{\ell})\in\mathcal{P}_{\ell}$ and $J = (j_{0},...,j_{\ell})\in\mathcal{P}_{\ell+1}$ we have
\begin{align}
& \intd^{\ell,*}(ue^{J}) = -\sum_{j = 1}^{k}e_{j}u\cdot\iota_{j}(e^{J}), \\
& \Delta_{\alpha,\ell}(ue^{I}) = \Delta_{\alpha}u\cdot e^{I}.
\end{align}
\end{lemma}
\begin{proof}
To simplify notation, let $\langle\cdot,\cdot\rangle_{\mu}$ denote the $L^{2}-$inner product induced by the measure $\mu$. By Lemma \ref{L:FormulaForDifferential} we obtain the formula
\begin{align*}
\left\langle\intd(ue^{I}),ve^{J}\right\rangle_{\ell+1} = & \left\langle\sum_{j = 1}^{k}e_{j}u\cdot e^{j}\wedge e^{I},ve^{J}\right\rangle_{\ell+1} = \sum_{j = 1}^{k}\langle e_{j}u,v\rangle_{\mu}\langle E_{j}(e^{I}),e^{J}\rangle_{\ell+1,0} = \\ = &
-\sum_{j = 1}^{k}\langle u,e_{j}v\rangle_{\mu}\langle e^{I},\iota_{j}(e^{J})\rangle_{\ell,0}
\end{align*}
where we used that $\langle e_{j}u,v\rangle_{\mu} = -\langle u,e_{j}v\rangle_{\mu}$ since $\mu$ is $\alpha-$invariant. That is, we have a formula for $\intd^{\ell,*}$:
\begin{align*}
\intd^{\ell,*}(ve^{J}) = -\sum_{j = 1}^{k}e_{j}v\cdot\iota_{j}(e^{J}),
\end{align*}
which proves the first part of the lemma. Let $I\in\mathcal{P}_{\ell}$ and $j\not\in I$, we define a number $\tau(j;I)\in\{0,1\}$ such that $(-1)^{\tau(j;I)}e^{j}\wedge e^{I} = e^{I\cup\{j\}}$. From the formula for $\intd^{\ell,*}$ we can calculate:
\begin{align*}
\intd^{\ell,*}\intd^{\ell}(ue^{I}) = & \intd^{\ell,*}\left(\sum_{j = 1}^{k}e_{j}u\cdot e^{j}\wedge e^{I}\right) = \sum_{j\not\in I}(-1)^{\tau(j;I)}\intd^{\ell,*}\left(e_{j}u\cdot e^{I\cup\{j\}}\right) = \\ = &
-\sum_{j\not\in I}(-1)^{\tau(j;I)}\sum_{t = 1}^{k}e_{t}e_{j}u\cdot\iota_{t}\left(e^{I\cup\{j\}}\right) = -\sum_{j,t = 1}^{k}e_{t}e_{j}u\cdot\iota_{t}\left(E_{j}(e^{I})\right)
\end{align*}
where we have used that $(-1)^{\tau(j;I)}e^{j}\wedge e^{I} = e^{I\cup\{j\}}$ so $\iota_{t}(E_{j}(e^{I})) = (-1)^{\tau(j;I)}\iota_{t}(e^{I\cup\{j\}})$. By a similar calculation we obtain
\begin{align}
&\label{Eq:HelpFormula1} \intd^{\ell,*}\intd^{\ell}(ue^{I}) = -\sum_{j,t = 1}^{k}e_{t}e_{j}u\cdot\iota_{t}\left(E_{j}(e^{I})\right), \\
&\label{Eq:HelpFormula2} \intd^{\ell-1}\intd^{\ell-1,*}(ue^{I}) = -\sum_{j,t = 1}^{k}e_{t}e_{j}u\cdot E_{t}(\iota_{j}(e^{I})).
\end{align}
Recall that our action is abelian, so we have $e_{j}e_{t}u = e_{t}e_{j}u$ for all $u\in C^{\infty}(M)$. By using formulas \ref{Eq:HelpFormula1} and \ref{Eq:HelpFormula2} we derive a formula for $\Delta_{\ell,\alpha}$:
\begin{align*}
\Delta_{\alpha,\ell}ue^{I} = & -\sum_{i,j = 1}^{k}e_{i}e_{j}u\cdot\left(\iota_{i}(E_{j}(e^{I})) + E_{i}(\iota_{j}(e^{I}))\right) = -\sum_{j = 1}^{k}e_{j}^{2}u\cdot\left(\iota_{j}(E_{j}(e^{I})) + E_{j}(\iota_{j}(e^{I})\right) - \\ &
\sum_{i < j}e_{i}e_{j}u\cdot\left(\iota_{i}(E_{j}(e^{I})) + E_{i}(\iota_{j}(e^{I})) + \iota_{j}(E_{i}(e^{I})) + E_{j}(\iota_{i}(e^{I}))\right).
\end{align*}
For $j\in I$ we have $E_{j}\iota_{j}e^{I} = e^{I}$ and $\iota_{j}E_{j}e^{I} = 0$, and if $j\not\in I$ then we have $E_{j}\iota_{j}e^{I} = 0$ and $\iota_{j}E_{j}e^{I} = e^{I}$ so we can simplify the formula for $\Delta_{\alpha,\ell}$ as
\begin{align*}
\Delta_{\alpha,\ell}ue^{I} = -\sum_{j = 1}^{k}e_{j}^{2}u\cdot e^{I} - \sum_{i < j}e_{i}e_{j}u\cdot\left(\iota_{i}(E_{j}(e^{I})) + E_{i}(\iota_{j}(e^{I})) + \iota_{j}(E_{i}(e^{I})) + E_{j}(\iota_{i}(e^{I}))\right).
\end{align*}
It remains to show $\iota_{i}(E_{j}(e^{I})) + E_{i}(\iota_{j}(e^{I})) + \iota_{j}(E_{i}(e^{I})) + E_{j}(\iota_{i}(e^{I})) = 0$ for $i < j$. If $i,j\not\in I$ or $i,j\in I$ then $\iota_{i}(E_{j}(e^{I})) + E_{i}(\iota_{j}(e^{I})) + \iota_{j}(E_{i}(e^{I})) + E_{j}(\iota_{i}(e^{I})) = 0$ since each term is $0$ separately. Assume instead that $j\in I$ and $i\not\in I$ (the last case is identical). In this case we note that $\iota_{i}(E_{j}(e^{I})) = 0$, $E_{j}(\iota_{i}(e^{I})) = 0$. Moreover, it is clear that $E_{i}\iota_{j}(e^{I})$ and $\iota_{j}E_{i}(e^{I})$ are proportional, formula \ref{Eq:AdjointOfContraction} then implies:
\begin{align*}
\langle E_{i}\iota_{j}(e^{I}),\iota_{j}E_{i}(e^{I})\rangle_{\ell,0} = \langle\iota_{i}E_{j}E_{i}\iota_{j}(e^{I}),e^{I}\rangle_{\ell,0} = -\langle\iota_{i}E_{i}E_{j}\iota_{j}(e^{I}),e^{I}\rangle_{\ell,0} = -\langle e^{I},e^{I}\rangle_{\ell,0} = -1
\end{align*}
so we obtain $E_{i}\iota_{j}(e^{I}) = -\iota_{j}E_{i}(e^{I})$. Equivalently, we have $E_{i}\iota_{j}(e^{I}) + \iota_{j}E_{i}(e^{I}) = 0$ finishing the proof.
\end{proof}
As a consequence of Lemma \ref{L:FormulaForChainLaplacian} we obtain a decomposition of the cochain spaces $C^{\ell}(\alpha)$ and can calculate the cohomology groups $H^{\ell}(\alpha)$ when the action $\alpha$ is ${\rm GH}$.
\begin{theorem}\label{Thm:DecompositionOfCochainSpaces}
Let $\alpha:\mathbb{R}^{k}\times M\to M$ be ${\rm GH}$. There is a (orthogonal with respect to $\langle\cdot,\cdot\rangle_{\ell}$) decomposition
\begin{align*}
C^{\ell}(\alpha) = {\rm Im}(\intd^{\ell-1})\oplus{\rm Im}(\intd^{\ell,*})\oplus\Lambda^{\ell}(\mathbb{R}^{k}),
\end{align*}
where we identify $\Lambda^{\ell}(\mathbb{R}^{k})\subset C^{\ell}(\alpha)$ by $\omega\mapsto\omega\otimes 1\in \Lambda^{\ell}(\mathbb{R}^{k})\otimes C^{\infty}(M)$ (or equivalently, an element $\omega\in\Lambda^{\ell}(\mathbb{R}^{k})$ is identified with the map taking $X_{1}\wedge...\wedge X_{\ell}\in\Lambda_{\ell}(\mathbb{R}^{k})$ to the constant function $\omega(X_{1}\wedge...\wedge X_{\ell})\in C^{\infty}(M)$). In particular, we obtain an isomorphism $H^{\ell}(\alpha)\cong\Lambda^{\ell}(\mathbb{R}^{k})$.
\end{theorem}
\begin{proof}
By definition we have $\text{Im}(\Delta_{\alpha,\ell})\subset\text{Im}(\intd^{\ell-1})\oplus\text{Im}(\intd^{\ell,*})$ (note that the sum $\text{Im}(\intd^{\ell-1})\oplus\text{Im}(\intd^{\ell,*})$ is direct since $\langle\intd^{\ell,*}\omega,\intd^{\ell-1}\eta\rangle_{\ell} = \langle\omega,\intd^{\ell}\intd^{\ell-1}\eta\rangle_{\ell+1} = 0$). Using the formula for $\Delta_{\alpha,\ell}$ in Lemma \ref{L:FormulaForChainLaplacian} and Theorem \ref{Thm:PropertiesOrbitLaplacian2} it is clear that ${\rm Im}(\Delta_{\alpha,\ell})\oplus\Lambda^{\ell}(\mathbb{R}^{k}) = C^{\ell}(\alpha)$. It follows that 
\begin{align*}
C^{\ell}(\alpha) = {\rm Im}(\Delta_{\alpha,\ell})\oplus\Lambda^{\ell}(\mathbb{R}^{k})\subset\text{Im}(\intd^{\ell-1})\oplus\text{Im}(\intd^{\ell,*})\oplus\Lambda^{\ell}(\mathbb{R}^{k})\subset C^{\ell}(\alpha),
\end{align*}
which shows $\text{Im}(\intd^{\ell-1})\oplus\text{Im}(\intd^{\ell,*})\oplus\Lambda^{\ell}(\mathbb{R}^{k}) = C^{\ell}(\alpha)$.

The calculation of $H^{\ell}(\alpha)$ follows since $\ker\intd^{\ell} = \text{Im}(\intd^{\ell-1})\oplus\Lambda^{\ell}(\mathbb{R}^{k})$. Indeed, $\ker\intd^{\ell}\supset\text{Im}(\intd^{\ell-1})\oplus\Lambda^{\ell}(\mathbb{R}^{k})$ is clear. To see the converse, note that if $\omega = \intd^{\ell,*}\eta\in\ker\intd^{\ell}$ then
\begin{align*}
\norm{\omega}_{\ell}^{2} = \norm{\intd^{\ell,*}\eta}_{\ell}^{2} = \langle\intd^{\ell,*}\eta,\intd^{\ell,*}\eta\rangle_{\ell} = \langle\eta,\intd^{\ell}\intd^{\ell,*}\eta\rangle_{\ell} = \langle\eta,\intd^{\ell}\omega\rangle_{\ell} = 0
\end{align*}
so $\omega = 0$, and ${\rm Im}(\intd^{\ell,*})\cap\ker\intd^{\ell} = 0$.
\end{proof}

\subsection{Betti number and invariant 1-forms}

In \cite{RordriguezHertzRodriguezHertz2006} F. Rodriguez Hertz and J. Rodriguez Hertz show that any cohomology free flow on a closed smooth manifold $M$ fiber over a linear flow on $\mathbb{T}^{b_{1}(M)}$, where $b_{1}(M)$ is the first Betti number of $M$. We use a variant of their proof to extend this result to more general actions of connected Lie groups. In particular, we show that GH $\mathbb{R}^{k}-$actions fibers over translation actions on tori.

The proof is based on the following observation: if $Y$ is a vector field on $M$ then for any closed $1-$form $\omega$ we have
\begin{align*}
\mathcal{L}_{Y}\omega = \intd u
\end{align*}
for some smooth function $u$. This is immediate from the formula of the Lie derivative on $1-$forms $\mathcal{L}_{Y} = \iota_{Y}\intd + \intd\iota_{Y}$, where $\iota_{Y}\theta = \theta(Y)$, (this is Cartan's formula, see for example \cite[Theorem 3.5.12.]{Mukherjee2015}) since $\iota_{Y}\omega$ is a smooth function if $\omega$ is a $1-$form. We will need the following standard lemma.
\begin{lemma}\label{L:IntegralCohomologyClassGivesCircleMap}
If $\omega$ is a closed $1-$form that projects to an integral cohomology class, then $\omega$ is given by $\intd s$ for some $s:M\to\mathbb{T}$.
\end{lemma}
\begin{proof}
Let $x_{0}\in M$ be some basepoint, then for $x\in M$ we define
\begin{align*}
s(x) := \int_{x_{0}}^{x}\omega + \mathbb{Z}.
\end{align*}
This map is well-defined since if $\gamma:\mathbb{T}\to M$ is a closed smooth curve then
\begin{align*}
\int_{\gamma}\omega = [\omega]([\gamma])\in\mathbb{Z}
\end{align*}
where $[\omega]([\gamma])$ is the pairing between the cohomology class $[\omega]$ and the homology class of $[\gamma]$. Since $[\omega]$ is integral this is an integer. It is clear that $s$ is smooth and $\intd s = \omega$.
\end{proof}
For a Lie algebra $\mathfrak{g}$ we denote by $H^{1}(\mathfrak{g})$ cohomology obtained from the trivial representation of $\mathfrak{g}$ on $\mathbb{C}$. If $\alpha:G\to{\rm Diff}^{\infty}(M)$ is a smooth $G-$action then we obtain a map $\Lambda^{\ell}(\mathfrak{g})\to C^{\ell}(\alpha)$ by mapping $\omega\in\Lambda^{\ell}(\mathfrak{g})$ to $\omega\otimes 1$ (where we consider $1$ as the constant function $1$ in $C^{\infty}(M)$) which induces an injective map $H^{1}(\mathfrak{g})\to H^{1}(\alpha)$. We can now state and prove the main theorem of this section.
\begin{theorem}\label{Thm:ComplicatedTopologyGiveSubmersion}
Let $G$ be any connected Lie group and let $\alpha:G\times M\to M$ be any minimal action with an invariant measure. If $H^{1}(\alpha) = H^{1}(\mathfrak{g})$, then there is a submersion $\pi:M\to\mathbb{T}^{b_{1}(M)}$ projecting $\alpha$ to an action by translations.
\end{theorem}
\begin{proof}
Let $b = b_{1}(M)$, and let $\omega\in\Omega^{1}(M)$ be a closed integral cohomology class. For any $X\in\mathfrak{g}$ we have
\begin{align*}
\mathcal{L}_{X}\omega = \intd\left[\omega(X)\right]
\end{align*}
by Cartan's formula. We define $\eta:\mathfrak{g}\to C^{\infty}(M)$ by $\eta(X)(x) = \omega_{x}(X) + c(X)$ where $c(X)\in\mathbb{R}$ is chosen such that $\eta(X)$ has zero average with respect to the invariant measure. For $X,Y\in\mathfrak{g}$ we have
\begin{align*}
0 = \mathcal{L}_{X}\mathcal{L}_{Y}\omega - \mathcal{L}_{Y}\mathcal{L}_{X}\omega - \mathcal{L}_{[X,Y]}\omega = \intd(X\eta(Y) - Y\eta(X) - \eta([X,Y])),
\end{align*}
so $X\eta(Y) - Y\eta(X) - \eta([X,Y]) = 0$ (since the left-hand side has integral $0$ with respect to the $\alpha-$invariant measure, and is constant). It follows that $\eta$ is a $1-$cocycle. Since $H^{1}(\alpha) = H^{1}(\mathfrak{g})$ there is a map $c:\mathfrak{g}\to\mathbb{R}$ and a function $u\in C^{\infty}(M)$ such that
\begin{align*}
\eta(X) = Xu + c(X),\quad X\in\mathfrak{g}.
\end{align*}
The functions $\eta(X)$ were chosen with zero average, so after integrating both sides of $\eta(X) = Xu + c(X)$ with respect to the $\alpha-$invariant measure, we obtain $c(X) = 0$. That is
\begin{align*}
\mathcal{L}_{X}(\omega - \intd u) = 0.
\end{align*}
So, after changing $\omega$ by a coboundary $\intd u$, we will assume without loss of generality that $\omega$ is $\alpha-$invariant. We claim that if $\omega$ represents a non-trivial element on cohomology, then $\omega$ is nowhere vanishing. Indeed, if $\omega_{x} = 0$ then
\begin{align*}
0 = \omega_{x} = (\alpha(g)^{*}\omega)_{x} = \omega_{\alpha(g)x}\circ D_{x}\alpha(g)
\end{align*}
or $\omega_{\alpha(g)x} = 0$. Since $\alpha$ is minimal this implies that $\omega$ vanishes identically, which contradicts that $\omega$ represents a non-trivial element on cohomology. 

Let $\omega_{1},...,\omega_{b}\in\Omega^{1}(M)$ be closed $1-$forms that represent integral cohomology classes, and projecting onto a basis of the $\mathbb{R}-$cohomology. Assume that each $\omega_{j}$ is $\alpha-$invariant. We claim that $(\omega_{1})_{x},...,(\omega_{b})_{x}$ are linearly independent for all $x\in M$. Indeed, if $c_{1},...,c_{b}$ are such that
\begin{align*}
c_{1}(\omega_{1})_{x} + ... + c_{b}(\omega_{b})_{x} = 0
\end{align*}
for some $x\in M$ then $c_{1}\omega_{1} + ... + c_{b}\omega_{b} = 0$ by the argument above. Since $\omega_{1},...,\omega_{b}$ forms a basis on cohomology, this implies that $c_{1} = ... = c_{b} = 0$ so $(\omega_{1})_{x},...,(\omega_{b})_{x}$ are indeed independent.

Let $s_{i}:M\to\mathbb{T}$, $i = 1,...,b$, be maps such that $\intd s_{i} = \omega_{i}$ (Lemma \ref{L:IntegralCohomologyClassGivesCircleMap}). Define
\begin{align*}
s:M\to\mathbb{T}^{b},\quad s = (s_{1},...,s_{b}).
\end{align*}
Since $\intd s_{i} = \omega_{i}$ and $(\omega_{1})_{x},...,(\omega_{b})_{x}$ are linearly independent at all $x\in M$, $s$ is a submersion. For any $X\in\mathfrak{g}$ we have $\intd(\intd s_{i}(X)) = \intd(\omega_{i}(X)) = \mathcal{L}_{X}\omega_{i} = 0$ or $\intd s_{i}(X) = c_{i}(X)$ is constant. That is, for $X\in\mathfrak{g}$ we have:
\begin{align*}
Ds(X) = (\intd s_{1}(X),...,\intd s_{b}(X)) = (c_{1}(X),...,c_{b}(X))
\end{align*}
proving that $s:M\to\mathbb{T}^{b}$ project $\alpha$ onto a translation action.
\end{proof}

\subsection{Tame estimates}

In local rigidity problems it is often useful to obtain tame estimates on the operators $\intd^{\ell}$. We prove that if $\Delta_{\alpha}$ is GH such that the inverse of $\Delta_{\alpha}$ (on $0-$average functions) satisfy tame estimates, then $\intd^{\ell}$ have an inverse on its image that satisfy tame estimates for all $\ell$. This will be used in forthcoming work to prove local rigidity of some parabolic actions on nilmanifolds \cite{Sandfeldt2024}.
\begin{theorem}
Let $\alpha:\mathbb{R}^{k}\times M\to M$ be a ${\rm GH}$ action. If the orbitwise laplacian has a tame inverse then every coboundary map $\intd^{\ell}:C^{\ell}(\alpha)\to C^{\ell + 1}(\alpha)$ has a tame inverse on its image. More precisely, there are tame maps $\delta^{\ell}:C^{\ell + 1}(\alpha)\to C^{\ell}(\alpha)$ such that $\intd^{\ell}\delta^{\ell}(\omega) = \omega$ for $\omega\in{\rm Im}(\intd^{\ell})$.
\end{theorem}
\begin{proof}
Recall that we have a decomposition $C^{\ell}(\alpha) = \text{Im}(\intd)\oplus\text{Im}(\intd^{*})\oplus\Lambda^{\ell}(\mathbb{R}^{k})$ with $\ker\intd^{\ell} = \text{Im}(\intd^{\ell-1})\oplus\Lambda^{\ell}(\mathbb{R}^{k})$ (Theorem \ref{Thm:DecompositionOfCochainSpaces}). So, the map $\intd^{\ell}:\text{Im}(\intd^{\ell,*})\to\text{Im}(\intd^{\ell})$ is bijective. Since each $H^{*}(\alpha)$ is finite dimensional it is immediate that $\text{Im}(\intd^{\ell})$ and $\text{Im}(\intd^{\ell,*})$ are both closed (note that we can write ${\rm Im}(\intd^{\ell})\oplus\Lambda^{\ell+1}(\mathbb{R}^{k}) = \ker(\intd^{\ell+1})$, where $\ker(\intd^{\ell+1})$ is closed and $\Lambda^{\ell+1}(\mathbb{R}^{k})$ is finite dimensional, similarly one sees that ${\rm Im}(\intd^{\ell,*})$ is closed). By the Open mapping theorem, we have a continuous inverse
\begin{align*}
\delta^{\ell}:\text{Im}(\intd^{\ell}:C^{\ell}(\alpha)\to C^{\ell+1}(\alpha))\to\text{Im}(\intd^{\ell,*}:C^{\ell+1}(\alpha)\to C^{\ell}(\alpha)).
\end{align*}
We want to use the fact that $\Delta_{\alpha}$ has a tame inverse to show that $\delta^{\ell}$ is a tame map. Since $\Delta_{\alpha,\ell}$ is $\Delta_{\alpha}$ coordinatewise (Lemma \ref{L:FormulaForChainLaplacian}) it follows that $\Delta_{\alpha,\ell}$ has a tame inverse on $\text{Im}(\intd^{\ell-1})\oplus\text{Im}(\intd^{\ell,*})$.

Note that the theorem is equivalent to showing that we have $C_{r}$ and $r_{0}$ such that for $\omega\in\text{Im}(\intd^{\ell-1,*})$ we have
\begin{align*}
\norm{\omega}_{r}\leq C_{r}\norm{\intd\omega}_{r + r_{0}}.
\end{align*}
This is implied by the following calculation
\begin{align*}
\norm{\omega}_{r} = & \norm{\Delta_{\alpha,\ell}^{-1}\Delta_{\alpha,\ell}\omega}_{r}\leq C_{r}'\norm{\Delta_{\alpha,\ell}\omega}_{r + r_{0}'} = C_{r}'\norm{(\intd^{\ell-1}\intd^{\ell-1,*} + \intd^{\ell,*}\intd^{\ell})\omega}_{r + r_{0}'} = \\ = &
C_{r}'\norm{\intd^{\ell,*}\intd^{\ell}\omega}_{r + r_{0}'}\leq C_{r}\norm{\intd^{\ell}\omega}_{r + r_{0}}
\end{align*}
where we used that $\omega\in\text{Im}(\intd^{\ell-1,*})$, so $\intd^{\ell,*}\omega = 0$, and that $\intd^{\ell,*}$ satisfies tame estimates since it is a degree $1$ differential operator.
\end{proof}

\section{Rigidity of globally hypoelliptic actions}

In this section, we prove Theorems \ref{MainThm:ThmA} and \ref{MainThm:ThmB}. We begin by proving $(i)$ of Theorem \ref{MainThm:ThmB} (and, as a special case also Theorem \ref{MainThm:ThmA}) in Section \ref{SubSec:GWConjectureOnNilmanifolds}. In Section \ref{SubSec:RigidityLargeFirstBettiNumber} we prove $(ii)$ of Theorem \ref{MainThm:ThmB}, this closely follows the proof of the Greenfield-Wallach conjecture on $3-$manifolds (when the first Betti number is $2$, see \cite[Section 5.1]{ForniGLWdim3}). Finally, in Section \ref{SubSec:RigiditySmallCodimension} we prove $(iii)$ of Theorem \ref{MainThm:ThmB}.

It will be convenient to first classify GH translation actions on tori, this is standard (see for example \cite[page 19]{KatokCombinatorial} or \cite[Section 3]{GreenfieldWallach2}) but we include a proof for completeness. Let $T:\mathbb{R}^{k}\times\mathbb{T}^{d}\to\mathbb{T}^{d}$ be an action by translations, and $\rho:\mathbb{R}^{k}\to\mathbb{R}^{d}$ the corresponding homomorphism such that $T(\mathbf{t})x = x + \rho(\mathbf{t})$. We write $X_{j} = \rho(e_{j})$, $j = 1,...,k$ (where $e_{j}$ is the $j$th standard basis vector). We say that $T$ (or $\rho$) is \textit{diophantine} if there are constants $K > 0$ and $\tau$ such that:
\begin{align}
\sum_{j = 1}^{k}\left|X_{j}\cdot\mathbf{n}\right|^{2}\geq\frac{K^{2}}{\norm{\mathbf{n}}^{2\tau}},\quad\mathbf{n}\in\mathbb{Z}^{d}\setminus0.
\end{align}
Our interest in diophantine translation actions is that a translation action is diophantine if and only if it is GH.
\begin{lemma}\label{L:GHactionOnTori}
Let $T:\mathbb{R}^{k}\times\mathbb{T}^{d}\to\mathbb{T}^{d}$ be a translation action. The action $T$ is ${\rm GH}$ if and only if $T$ is diophantine.
\end{lemma}
\begin{proof}
One direction follows since the diophantine assumption implies that we can solve the \textit{small divisor problem}. Let $X_{j} = \rho(e_{j})$, $j = 1,...,k$, be the generators of $T$. For $\mathbf{n}\in\mathbb{Z}^{d}$:
\begin{align}
\Delta_{T}e^{2\pi i\mathbf{n}\cdot x} := -(X_{1}^{2}+...+X_{k}^{2})e^{2\pi i\mathbf{n}\cdot x} = 4\pi^{2}\sum_{j = 1}^{k}|X_{j}\cdot\mathbf{n}|^{2}.
\end{align}
So, given $v\in C^{\infty}(\mathbb{T}^{d})$, we write
\begin{align}
v(x) = \sum_{\mathbf{n}\in\mathbb{Z}^{d}}\Hat{v}(\mathbf{n})e^{2\pi i\mathbf{n}\cdot x} = \sum_{\mathbf{n}\in\mathbb{Z}^{d}\setminus0}\Delta_{T}\left[\frac{\Hat{v}(\mathbf{n})}{4\pi^{2}\sum_{j = 1}^{k}|X_{j}\cdot\mathbf{n}|^{2}}e^{2\pi i\mathbf{n}\cdot x}\right] + \Hat{v}(0).
\end{align}
Define
\begin{align}
u(x) = \sum_{\mathbf{n}\in\mathbb{Z}^{d}\setminus0}\frac{\Hat{v}(\mathbf{n})}{4\pi^{2}\sum_{j = 1}^{k}|X_{j}\cdot\mathbf{n}|^{2}}e^{2\pi i\mathbf{n}\cdot x},
\end{align}
then $\Delta_{T}u = v - \Hat{v}(0)$ (as a formal series). The diophantine condition implies that the formal series $u$ defines a smooth function, so $T$ is GH (by Theorem \ref{Thm:PropertiesOrbitLaplacian2}). Conversely, we want to show that if $T$ is GH then $T$ satisfies a diophantine condition. By Theorem \ref{Thm:PropertiesOrbitLaplacian2} and the open mapping theorem, the operator $\Delta_{T}:C_{0}^{\infty}(\mathbb{T}^{d})\to C_{0}^{\infty}(\mathbb{T}^{d})$ has a continuous inverse in the Fréchet topology. That is, for $r\in\mathbb{N}$ there is $s\in\mathbb{N}$ such that $\norm{\Delta_{T}u}_{s}\geq C_{r}\norm{u}_{r}$ for all $u\in C_{0}^{\infty}(\mathbb{T}^{d})$. In particular, there is $\tau$ such that for any $\mathbf{n}\in\mathbb{Z}^{d}\setminus0$
\begin{align}
(2\pi)^{2\tau}\norm{\mathbf{n}}^{2\tau}\cdot(2\pi)^{2}\sum_{j = 1}^{k}|X_{j}\cdot\mathbf{n}|^{2} = \norm{\Delta_{T}e^{2\pi i\mathbf{n}\cdot\mathbf{x}}}_{2\tau}\geq C_{0}\norm{e^{2\pi i\mathbf{n}\cdot\mathbf{x}}}_{0} = C_{0},
\end{align}
which shows, after rearranging terms, that $T$ is diophantine.
\end{proof}

\subsection{Proof of Theorem \ref{MainThm:ThmB} (i): The Greenfield-Wallach conjecture on nilmanifolds}\label{SubSec:GWConjectureOnNilmanifolds}

Let $G$ be a simply connected nilpotent lie group and $\Gamma\leq G$ a lattice. In this section, we prove $(i)$ in Theorem \ref{MainThm:ThmB} following \cite{Kocsard2007}. As a corollary, we also prove Theorem \ref{MainThm:ThmA}. The proof of Theorem \ref{MainThm:ThmB} $(i)$ follows from a general result on cocycles taking values in nilpotent groups. Similar ideas have been used, for example, in \cite{Rodriguez-HertzWang2014}.
\begin{lemma}\label{L:NilpotentCocycleRigidity}
Let $G$ be a simply connected nilpotent Lie group and $\alpha:\mathbb{R}^{k}\times M\to M$ a smooth ${\rm GH}$ action. If $\gamma:\mathbb{R}^{k}\times M\to G$ is a smooth $G-$valued cocycle then $\gamma$ is cohomologous to a constant cocycle.
\end{lemma}
\begin{remark}
It suffices to assume that $\alpha$ satisfies $H^{1}(\alpha) = \Lambda^{1}(\mathbb{R}^{k})\cong\mathbb{R}^{k}$ and that $\alpha$ is minimal, so the assumption that $\alpha$ is GH can be weakened.
\end{remark}
\begin{proof}
We will prove the theorem by induction on the step of $G$. If $G$ is $1-$step then $G = \mathbb{R}^{n}$ for some $n$. The lemma follows from Theorems \ref{Thm:PropertiesOrbitLaplacian2} and \ref{Thm:Cohomology1} since $\alpha$ is assumed to be GH.

Assume that the theorem holds when $G$ is step $\ell-1$, with $\ell > 1$. Given a cocycle $\gamma:\mathbb{R}^{k}\times M\to G$ let $\Tilde{\gamma}$ be the $G/G_{(\ell)}-$valued cocycle defined by $\Tilde{\gamma}:\mathbb{R}^{k}\times M\to G\to G/G_{(\ell)}$. Since $G/G_{(\ell)}$ has step $\ell-1$, we find $\Tilde{b}:M\to G/G_{(\ell)}$ and $\Tilde{\rho}:\mathbb{R}^{k}\to G/G_{(\ell)}$ such that $\Tilde{\gamma}(\mathbf{t},x) = \Tilde{b}(\alpha(\mathbf{t})x)^{-1}\Tilde{\rho}(\mathbf{t})\Tilde{b}(x)$. Let $S:G/G_{(\ell)}\to G$ be any smooth section, and let $\widehat{b} = S\circ\Tilde{b}$, $\widehat{\rho}(\mathbf{t}) = S\circ\Tilde{\rho}(\mathbf{t})$. It is immediate
\begin{align*}
\widehat{b}(\alpha(\mathbf{t})x)\gamma(\mathbf{t},x)\widehat{b}(x)^{-1}\widehat{\rho}(\mathbf{t})^{-1}\in G_{(\ell)}
\end{align*}
since if we project it by the map $G\to G/G_{(\ell)}$ then we get identity. Define $\rho(\mathbf{t})$ by
\begin{align}
\rho(\mathbf{t}) := \widehat{\rho}(\mathbf{t})\int_{M}\widehat{b}(\alpha(\mathbf{t})x)\gamma(\mathbf{t},x)\widehat{b}(x)^{-1}\widehat{\rho}(\mathbf{t})^{-1}\intd\mu(x)
\end{align}
where $\mu$ is the unique invariant measure (Theorem \ref{Thm:PropertiesOrbitLaplacian2}), and the integral makes sense since $G_{(\ell)}\cong\mathbb{R}^{d_{\ell}}$ for some integer $d_{\ell}$. Let
\begin{align}
\widehat{b}(\alpha(\mathbf{t})x)\gamma(\mathbf{t},x)\widehat{b}(x)^{-1} = \widehat{\gamma}(\mathbf{t},x).
\end{align}
Using that $G_{(\ell)}$ is central in $G$, and noting that $\Hat{\rho}(\mathbf{t}+\mathbf{s})\Hat{\rho}(\mathbf{t})^{-1}\Hat{\rho}(\mathbf{s})^{-1}\in G_{(\ell)}$, we obtain
\begin{align*}
\rho(\mathbf{t}+\mathbf{s}) = & \widehat{\rho}(\mathbf{t}+\mathbf{s})\int_{M}\widehat{\gamma}(\mathbf{t}+\mathbf{s},x)\widehat{\rho}(\mathbf{t}+\mathbf{s})^{-1}\intd\mu(x) = \\ = &
\widehat{\rho}(\mathbf{t}+\mathbf{s})\int_{M}\widehat{\gamma}(\mathbf{t},\alpha(\mathbf{s})x)\widehat{\gamma}(\mathbf{s},x)\widehat{\rho}(\mathbf{t}+\mathbf{s})^{-1}\intd\mu(x) = \\ = &
\widehat{\rho}(\mathbf{t}+\mathbf{s})\int_{M}\widehat{\gamma}(\mathbf{t},\alpha(\mathbf{s})x)\widehat{\rho}(\mathbf{t})^{-1}\widehat{\rho}(\mathbf{t})\widehat{\gamma}(\mathbf{s},x)\widehat{\rho}(\mathbf{s})^{-1}\widehat{\rho}(\mathbf{s})\cdot \\ \cdot & 
\widehat{\rho}(\mathbf{t}+\mathbf{s})^{-1}\intd\mu(x) = \\ = &
\left(\widehat{\rho}(\mathbf{t})\int_{M}\widehat{\gamma}(\mathbf{t},\alpha(\mathbf{s})x)\widehat{\rho}(\mathbf{t})^{-1}\intd\mu(x)\right)\cdot \\ \cdot &
\left(\widehat{\rho}(\mathbf{s})\int_{M}\widehat{\gamma}(\mathbf{s},x)\widehat{\rho}(\mathbf{s})^{-1}\intd\mu(x)\right) = \\ = &
\rho(\mathbf{t})\rho(\mathbf{s})
\end{align*}
where we have used that $\int_{M}a(x)b(x)\intd\mu(x) = \int_{M}a(x)\intd\mu(x)\int_{M}b(x)\intd\mu(x)$ for $a,b:M\to G_{(\ell)}$ (we are using multiplicative notation for $G_{(\ell)}$, the equality is clear in additive notation). That is, $\rho:\mathbb{R}^{k}\to G$ is a homomorphism. Let $Z:M\to G^{\ell}$, and define $b(x) = \widehat{b}(x)Z(x)$. We want to choose $Z$ such that $b(\alpha(\mathbf{t})x)^{-1}\gamma(\mathbf{t},x)b(x) = \rho(\mathbf{t})$. Define $\eta:\mathbb{R}^{k}\times M\to G_{(\ell)}$ by
\begin{align*}
\eta(\mathbf{t},x) = \widehat{b}(\alpha(\mathbf{t})x)\gamma(\mathbf{t},x)\widehat{b}(x)^{-1}\rho(\mathbf{t})^{-1} = \widehat{\gamma}(\mathbf{t},x)\rho(\mathbf{t})^{-1}.
\end{align*}
Note that
\begin{align*}
\eta(\mathbf{t}+\mathbf{s},x) = & \widehat{\gamma}(\mathbf{t}+\mathbf{s},x)\rho(\mathbf{t}+\mathbf{s})^{-1} = \widehat{\gamma}(\mathbf{t},\alpha(\mathbf{s})x)\widehat{\gamma}(\mathbf{s},x)\rho(\mathbf{s})^{-1}\rho(\mathbf{t})^{-1} = \\ = &
\widehat{\gamma}(\mathbf{t},\alpha(\mathbf{s})x)\rho(\mathbf{t})^{-1}\widehat{\gamma}(\mathbf{s},x)\rho(\mathbf{s})^{-1} = \eta(\mathbf{t},\alpha(\mathbf{s})x)\eta(\mathbf{s},x)
\end{align*}
where we used that $\widehat{\gamma}(\mathbf{s},x)\rho(\mathbf{s})^{-1} = \eta(\mathbf{s},x)\in G_{(\ell)}$ is central. It follows that $\eta$ is a $G_{(\ell)}-$valued cocycle. The definition of $\rho$ also implies that $\int_{M}\eta(\mathbf{t},x)\intd\mu(x) =e$ for all $\mathbf{t}\in\mathbb{R}^{k}$. Since $H^{1}(\alpha) = \Lambda^{1}(\mathbb{R}^{k})\cong\mathbb{C}^{k}$ (by Theorems \ref{Thm:PropertiesOrbitLaplacian2} and \ref{Thm:Cohomology1}) we find $Z:M\to G_{(\ell)}$ such that $\eta(\mathbf{t},x) = Z(\alpha(\mathbf{t})x)^{-1}Z(x)$. The lemma follows if we let $b(x) = \widehat{b}(x)Z(x)$.
\end{proof}
\begin{proof}[Proof of Theorem \ref{MainThm:ThmB}, (i)]
Let $\alpha:\mathbb{R}^{k}\times M_{\Gamma}\to M_{\Gamma}$ be ${\rm GH}$ where $M_{\Gamma} = \Gamma\setminus G$ is a compact nilmanifold. Define $\gamma:\mathbb{R}^{k}\times M_{\Gamma}\to G$ by
\begin{align*}
\alpha(\mathbf{t})x = x\left[\gamma(\mathbf{t},x)\right]^{-1}.
\end{align*}
The map $\gamma$ is uniquely defined if $\gamma(0,x) = e$. Given $\mathbf{t},\mathbf{s}\in\mathbb{R}^{k}$ we have
\begin{align*}
\alpha(\mathbf{t} + \mathbf{s})x = & x\left[\gamma(\mathbf{t} + \mathbf{s},x)\right]^{-1} = \alpha(\mathbf{t})\alpha(\mathbf{s})x = \left[\alpha(\mathbf{s})x\right]\cdot\left[\gamma(\mathbf{t},\alpha(\mathbf{s})x)\right]^{-1} \\ &
x\left[\gamma(\mathbf{s},x)\right]^{-1}\cdot\left[\gamma(\mathbf{t},\alpha(\mathbf{s})x)\right]^{-1} = x\left[\gamma(\mathbf{t},\alpha(\mathbf{s})x)\gamma(\mathbf{s},x)\right]^{-1}
\end{align*}
so, $\gamma$ is a $G-$valued cocycle over $\alpha$. By Lemma \ref{L:NilpotentCocycleRigidity}, the cocycle $\gamma$ is cohomologous to a constant cocycle. That is, we find a homomorphism $\rho:\mathbb{R}^{k}\to G$ and some $b:M_{\Gamma}\to G$ such that $\gamma(\mathbf{t},x) = b(\alpha(\mathbf{t})x)^{-1}\rho(\mathbf{t})b(x)$. Define $H:M_{\Gamma}\to M_{\Gamma}$ by $H(x) = xb(x)^{-1}$. Then
\begin{align*}
H(\alpha(\mathbf{t})x) = & \left[\alpha(\mathbf{t})x\right]b(\alpha(\mathbf{t})x)^{-1} = x\left[b(\alpha(\mathbf{t})x)^{-1}\rho(\mathbf{t})b(x)\right]^{-1}b(\alpha(\mathbf{t})x)^{-1} = \\ &
xb(x)^{-1}\rho(-\mathbf{t}) = H(x)\rho(-\mathbf{t}).
\end{align*}
So, $H$ conjugates $\alpha$ to the algebraic action induced by $\rho:\mathbb{R}^{k}\to G$. Since $H$ is smooth, we define $K := \{\det(D_{x}H) = 0\}\subset M_{\Gamma}$ which is compact and $\alpha-$invariant. The minimality of $\alpha$ implies $K = \emptyset$ or $K = M_{\Gamma}$, but the latter possibility is excluded by Sard's theorem. We conclude $K = \emptyset$. That is, $H:M_{\Gamma}\to M_{\Gamma}$ is a surjective (since it is homotopic to identity) local diffeomorphism. Finally, since $H$ is homotopic to identity $H$ has degree $1$ so the fibers of $H$ contain only one point, so $H$ is a diffeomorphism.
\end{proof}
We proceed to the proof of Theorem \ref{MainThm:ThmA}. By Theorem \ref{MainThm:ThmB} $(i)$, any GH flow on a nilmanifold $M_{\Gamma}$ is a homogeneous nilflow. By \cite{ForniRHGWC} (or \cite{ForniCohEq}) this finishes the proof of Theorem \ref{MainThm:ThmA}. However, we will give alternative proof (that nilflows are only GH if they are diophantine flow on tori) since we will need some of these results in the proof of $(ii)$ and $(iii)$ in Theorem \ref{MainThm:ThmB}. Recall that $H^{g}$ is the $g$th Heisenberg group (Example \ref{Ex:HeisenberGroup}).
\begin{lemma}\label{L:GHonHeisenberg}
Let $\Gamma\leq H^{g}$ be a lattice, $M_{\Gamma}$ the associated nilmanifold, $\rho:\mathbb{R}^{k}\to H^{g}$ a homomorphism, and $\alpha(\mathbf{t})x = x\rho(\mathbf{t})$ be the associated translation action on $M_{\Gamma}$. If $\alpha$ is GH then ${\rm Im}(\rho)\cap[H^{g},H^{g}] = [H^{g},H^{g}]$ and the projected translation action on $M_{\Gamma}/[H^{g},H^{g}]\cong\mathbb{T}^{2g}$ is diophantine.
\end{lemma}
\begin{proof}
The projected action on $\mathbb{T}^{2g}$ is diophantine by Lemmas \ref{L:PropertiesOrbitLaplacian2} and \ref{L:GHactionOnTori}. Assume that ${\rm Im}(\rho)\cap[H^{g},H^{g}] = e$. Let $X_{1},...,X_{g},Y_{1},...,Y_{g},Z\in\mathfrak{h}^{g}$ be such that $[X_{i},Y_{i}] = Z$. Assume, without loss of generality, that $X_{1},...,X_{k}$ span ${\rm Im}(D\rho)$ (so $X_{1},...,X_{k}$ generate $\alpha$). Let $\pi_{\hbar}$, $\hbar\neq0$, be any Schrödinger representation on $L^{2}(\mathbb{R}^{g})$ (see \cite[Chapter 10]{HarmonicAnalysis} or \cite[Example 2.2.6]{CorwinGreenleaf}). The operator $\Delta_{\alpha}$ can be written in the representation $\pi_{\hbar}$ as
\begin{align*}
u(x_{1},...,x_{g})\mapsto 4\pi^{2}(x_{1}^{2}+...+x_{k}^{2})\cdot u(x_{1},...,x_{g}).
\end{align*}
It follows that any smooth vector $u\in L^{2}(\mathbb{R}^{g})$ (which coincide with Schwartz function, \cite[Corollary 4.1.2]{CorwinGreenleaf}) that lie in the image of $\Delta_{\alpha}$ satisfy $u(0) = 0$. But, $e^{-|x|^{2}}\in L^{2}(\mathbb{R}^{g})$ is a smooth vector that does not vanish at $0$. So, $\Delta_{\alpha}$ can not be surjective on the representation $\pi_{\hbar}$. Since $L^{2}(M_{\Gamma})$ contain a Schrödinger representation, it follows that $\Delta_{\alpha}:C_{0}^{\infty}(M_{\Gamma})\to C_{0}^{\infty}(M_{\Gamma})$ is not surjective. By Theorem \ref{Thm:PropertiesOrbitLaplacian2} $\alpha$ is not GH.
\end{proof}
We can now finish the proof of Theorem \ref{MainThm:ThmB}.
\begin{proof}[Proof of Theorem \ref{MainThm:ThmB} (i)]
Let $\phi_{t}$ be a GH flow. By Theorem \ref{MainThm:ThmB} $(i)$ it suffices to consider the case when the flow $\phi_{t}$ is homogeneous. Note that $M_{\Gamma}$ fibers over $\Lambda\setminus H^{g}$ for some $g$, where $g\neq0$ unless $M_{\Gamma}$ is a torus (by quotiening out $G^{(3)}$ we obtain a $2-$step nilpotent nilmanifold and when $G$ is $2-$step we can quotient out a codimension one rational subgroup of the derived subgroup to obtain a Heisenberg group times a torus, finally we quotient out the torus). Let $\Tilde{\phi}_{t}$ be the induced flow on $\Lambda\setminus H^{g}$, Lemma \ref{L:PropertiesOrbitLaplacian2} implies that $\Tilde{\phi}_{t}$ is GH. By Lemma \ref{L:GHonHeisenberg} the action $\Tilde{\phi}_{t}$ intersect the center $[H^{g},H^{g}]$, but this is a contradiction since the projected flow on the base torus $\mathbb{T}^{2g}\cong\Lambda\setminus H^{g}/[H^{g},H^{g}]$ must be diophantine (Lemma \ref{L:GHactionOnTori}).
\end{proof}

\subsection{Proof of Theorem \ref{MainThm:ThmB} (ii): Manifolds with large first Betti number}\label{SubSec:RigidityLargeFirstBettiNumber}

Let $\alpha:\mathbb{R}^{k}\times M\to M$ be a GH action on $M$ where $b_{1}(M)\geq d-1$, $\dim M = d$. We now start the proof of Theorem \ref{MainThm:ThmB} $(ii)$. By Theorem \ref{Thm:ComplicatedTopologyGiveSubmersion} there is a submersion $\pi:M\to\mathbb{T}^{b_{1}(M)}$. This immediately implies $b_{1}(M)\leq\dim(M)$, and if $b_{1}(M) = \dim(M)$, then $M$ is a torus. If $M$ is a torus, then $(i)$ in Theorem \ref{MainThm:ThmB} finishes the proof of $(ii)$ in Theorem \ref{MainThm:ThmB} (alternatively we can apply \cite[Theorem 2.5]{Kocsard2007}). Suppose instead $b_{1}(M) = d - 1$.
\begin{lemma}\label{L:ReductionFromLargeBettiTohomogeneous}
If $b_{1}(M) = d-1$ then there is a simply connected $2-$step nilpotent Lie group $G$ with $\dim[\mathfrak{g},\mathfrak{g}] = 1$ such that $M = \Gamma\setminus G$ and the action $\alpha$ is homogeneous on $\Gamma\setminus G$.
\end{lemma}
\begin{proof}
Let $X_{1},...,X_{k}\in\mathbb{R}^{k}$ be generators of $\alpha$. Let $s:M\to\mathbb{T}^{d-1}$ be the submersion from Theorem \ref{Thm:ComplicatedTopologyGiveSubmersion}, and let $\omega_{i} = s^{*}\intd x_{i}\in\Omega^{1}(M)$. Since $s$ semi-conjugate $\alpha$ onto a translation action, each $\omega_{i}$ is $\alpha-$invariant. Define $Z\in\Gamma(TM)$ as the unique vector field satisfying $\iota_{Z}\mu = \omega_{1}\wedge...\wedge\omega_{d-1}$ where $\mu$ is the $\alpha-$invariant volume form (Theorem \ref{Thm:PropertiesOrbitLaplacian1}). Let $\nu$ be the Haar measure on $\mathbb{T}^{d-1}$, then $s^{*}\nu = \omega_{1}\wedge...\wedge\omega_{d-1} = \iota_{Z}\mu$. For each $X_{j}$ we have
\begin{align*}
\iota_{[X_{j},Z]}\mu = [\mathcal{L}_{X_{j}},\iota_{Z}]\mu = \mathcal{L}_{X_{j}}\iota_{Z}\mu - \iota_{Z}\mathcal{L}_{X_{j}}\mu = 0
\end{align*}
since $\omega_{1}\wedge...\wedge\omega_{d-1}$ and $\mu$ are both $\alpha-$invariant. The map $T_{x}M\ni Y\mapsto\iota_{Y}\mu$ is an isomorphism since $\mu$ is a volume form. It follows that $[X_{j},Z] = 0$, so $Z$ commute with the action $\alpha$. Moreover, if $v_{1},...,v_{d-2}\in T_{x}M$ then
\begin{align*}
0 = & (\iota_{Z}\iota_{Z}\mu)(v_{1},...,v_{d-2}) = (\iota_{Z}s^{*}\nu)(v_{1},...,v_{d-2}) = s^{*}\nu(Z,v_{1},...,v_{d-2}) = \\ = &
\nu(Ds(Z),Ds(v_{1}),...,Ds(v_{d-2})).
\end{align*}
Since we can choose $v_{1},...,v_{d-2}$ freely, $\nu$ is a volume form on $\mathbb{T}^{d-1}$, and $s$ is a submersion it follows that $Ds(Z) = 0$, which implies $\ker Ds = \mathbb{R}Z$.

Let $Y_{1},...,Y_{d-1}$ be vector fields that cover $\partial_{x^{1}},...,\partial_{x^{d-1}}$ on $\mathbb{T}^{d-1}$ under $s$ ($Ds(Y_{j}) = \partial_{x^{j}}$). For all $i = 1,...,k$ and $j = 1,...,d-1$, 
\begin{align}
Ds([X_{i},Y_{j}]) = [Ds(X_{i}),Ds(Y_{j})] = 0
\end{align}
since both $Ds(Y_{j})$ and $Ds(X_{i})$ are linear. With $j$ fixed, let $[X_{i},Y_{j}] = v_{i}\cdot Z$ where $v_{i}\in C^{\infty}(M)$. We claim that $X_{i}v_{\ell} = X_{\ell}v_{i}$, so the functions $v_{i}$ form an element of $Z^{1}(\alpha)$ (formula \ref{Eq:SecondCoboundaryOperator}). Indeed, the Jacobi identity implies
\begin{align*}
X_{\ell}v_{i}\cdot Z = & [X_{\ell},[X_{i},Y_{j}]] = -\left([Y_{j},[X_{\ell},X_{i}]] + [X_{i},[Y_{j},X_{\ell}]]\right) = \\ = &
[X_{i},[X_{\ell},Y_{j}]] = \mathcal{L}_{X_{i}}(v_{\ell}\cdot Z) = X_{i}v_{\ell}\cdot Z + v_{\ell}\cdot[X_{i},Z] = \\ = &
X_{i}v_{\ell}\cdot Z.
\end{align*}
Theorem \ref{Thm:DecompositionOfCochainSpaces} implies that there is $u_{j}$ such that $v_{i} = X_{i}u_{j} + c_{i}$ for $i = 1,...,k$, where $c_{i}\in\mathbb{R}$. Define $\Tilde{Y}_{j} = Y_{j} - u_{j}\cdot Z$, then $[X_{i},\Tilde{Y}_{j}] = -c_{i}Z$. That is, we may assume without loss of generality that $Y_{1},...,Y_{d-1}$ are chosen such that $[X_{i},Y_{j}]\in\mathbb{R}Z$ for all $i,j$, we will do this in the remainder. Note that $Ds[Y_{j},Z] = 0$ since $Ds(Z) = 0$, so $[Y_{j},Z] = f\cdot Z$. But also
\begin{align*}
X_{i}f\cdot Z = \mathcal{L}_{X_{i}}[Y_{j},Z] = -[[X_{i},Y_{j}],Z] - [Y_{j},[X_{i},Z]] = 0,
\end{align*}
since $[X_{i},Y_{j}]\in\mathbb{R}Z$ and $[X_{i},Z] = 0$. Minimality of $\alpha$ implies that $f$ is constant, so $[Y_{j},Z] = \lambda_{j}Z$ for some $\lambda_{j}\in\mathbb{R}$. Again, $Ds[Y_{i},Y_{j}] = 0$ so $[Y_{i},Y_{j}] = u\cdot Z$ for some $u\in C^{\infty}(M)$. Differentiating along $X_{i}$, and using the Jacobi identity we obtain $X_{i}u\cdot Z = 0$. It follows, again, that $u$ is constant so $[Y_{i},Y_{j}] = c_{ij}Z$ for some $c_{ij}\in\mathbb{R}$.

Let $\mathfrak{g}$ be the Lie algebra generated by $Y_{1},...,Y_{d},Z$, with brackets $[Y_{i},Y_{j}] = c_{ij}Z$, and $[Y_{i},Z] = \lambda_{i}Z$. Let $G$ be the associated simply connected Lie group. Since $Y_{1},...,Y_{d-1},Z$ form a frame of $M$, $G$ acts locally freely on $M$ and the stabilizer of this action, $\Gamma\leq G$, is a lattice in $G$ since $M$ is compact. That is, we can write $M\cong\Gamma\setminus G$. Moreover, with this identification $\alpha$ acts by translations. Since $G$ admits a lattice, it is unimodular (see for example \cite[Theorem 9.1.6]{HarmonicAnalysis}). If $W\in\mathfrak{g}$ and $f_{W}:\Gamma\setminus G\to\Gamma\setminus G$ is given by right-translation of $e^{-W}$, then $Df_{W}$ is naturally identified with ${\rm Ad}_{e^{W}} = e^{{\rm ad}_{W}}$. The group $G$ is unimodular, so $f_{W}$ preserves volume which implies $1 = \det(e^{{\rm ad}_{W}}) = e^{{\rm tr}({\rm ad}_{W})}$ or ${\rm tr}({\rm ad}_{W}) = 0$. On the other hand, $\text{ad}_{W}(Z) = \eta(W)\cdot Z$ with $\eta:\mathfrak{g}\to\mathbb{R}$ linear. So, if we calculate the trace in the basis $Y_{1},...,Y_{d-1},Z$ then we have $\text{tr}(\text{ad}_{W}) = \eta(Z) + 0 = 0$ (where ${\rm ad}_{W}(Y_{j})\in\mathbb{R}Z$ so these terms do not contribute to the trace). That is: $\text{ad}_{W}(Z) = [W,Z] = 0$. It follows that $Z$ is central. This implies the lemma if $\mathfrak{g}$ is non-abelian since $Z$ is central and every bracket is contained in $\mathbb{R}Z$. If $\mathfrak{g}$ is abelian then $\Gamma\setminus G$ is a torus of dimension $\dim(M)$ which would imply $b_{1}(M) = \dim(M)$, a contradiction.
\end{proof}
The following lemma is standard.
\begin{lemma}\label{L:ClassificationOf2StepNilpotent}
Any simply connected $2-$step nilpotent Lie group with $\dim[G,G] = 1$ is $H^{g}\times\mathbb{R}^{n}$ for some $g\geq1$ and $n\geq0$.
\end{lemma}
Using this lemma we can finish the proof of Theorem \ref{MainThm:ThmB} $(ii)$.
\begin{proof}[Proof of Theorem \ref{MainThm:ThmB} $(ii)$]
By Lemmas \ref{L:ReductionFromLargeBettiTohomogeneous} and \ref{L:ClassificationOf2StepNilpotent} Theorem \ref{MainThm:ThmB} $(ii)$ follows if $b_{1}(M) = d-1$, in this case Lemma \ref{L:GHonHeisenberg} also shows that $k\geq 2$ and the translation action intersect the derived subgroup. If $b_{1}(M) = d$ then Theorem \ref{MainThm:ThmB} $(ii)$ follows from Theorem \ref{MainThm:ThmB} $(i)$ and Theorem \ref{Thm:ComplicatedTopologyGiveSubmersion}. Theorem \ref{Thm:ComplicatedTopologyGiveSubmersion} also shows that $b_{1}(M)\leq d$, which finishes the proof.
\end{proof}

\subsection{Proof of Theorem \ref{MainThm:ThmB} point (iii): Actions with small codimension}\label{SubSec:RigiditySmallCodimension}

Let $\alpha:\mathbb{R}^{k}\times M\to M$ be a GH action on $M$, where $\dim M = d = k + 1$. In this section, we prove Theorem \ref{MainThm:ThmB} $(iii)$. We begin by showing that $M$ is diffeomorphic to $\Gamma\setminus G$ with $G$ some solvable Lie group and $\Gamma$ a lattice. We then show that $G$ is necessarily nilpotent (and, in fact, isomorphic to $\mathbb{R}^{n}$ or $H^{1}\times\mathbb{R}^{n}$).

Let $\mathcal{O}_{\alpha}$ be the orbit foliation of $\alpha$ and $\mu\in\Omega^{d}(M)$ the $\alpha-$invariant volume form (which exists by Lemma \ref{Thm:PropertiesOrbitLaplacian1}). Let $X_{1},...,X_{k}$ be generators of $\alpha$. Defining $\theta\in\Omega^{1}(M)$ by $\theta = \iota_{X_{k}}\iota_{X_{k-1}}...\iota_{X_{1}}\mu$, it is immediate that $\theta$ is $\alpha-$invariant and that $\theta$ vanishes on $T\mathcal{O}_{\alpha}$. Let $Z\in\Gamma(TM)$ be such that $\theta(Z) = 1$. Since $\ker\theta = T\mathcal{O}_{\alpha}$ it follows that $TM = T\mathcal{O}_{\alpha}\oplus\mathbb{R}Z$. Let $Y = a_{1}X_{1}+...+a_{d-1}X_{k}\in\mathbb{R}^{k}$, then
\begin{align*}
0 = \mathcal{L}_{Y}(\theta(Z)) = \mathcal{L}_{Y}(\theta)(Z) + \theta(\mathcal{L}_{Y}(Z)) = \theta(\mathcal{L}_{Y}(Z))
\end{align*}
so $\mathcal{L}_{Y}(Z)\in\ker\theta = T\mathcal{O}_{\alpha}$, which implies $\mathcal{L}_{Y}(Z) = f_{1}X_{1}+...+f_{k}X_{k}$ with $f_{1},...,f_{k}\in C^{\infty}(M)$.
\begin{lemma}\label{L:ExistenceOfHomogeneous}
We can choose $Z$ such that $[Z,Y]\in\mathbb{R}^{k}$ for all $Y\in\mathbb{R}^{k}$. That is, we can choose $Z$ such that $\mathbb{R}X_{1}+...+\mathbb{R}X_{k} + \mathbb{R}Z$ is a Lie algebra.
\end{lemma}
\begin{proof}
Let $Y\in\mathbb{R}^{k}$. Define maps $Y\mapsto f_{i}(Y)\in C^{\infty}(M)$ by $\mathcal{L}_{Y}(Z) = f_{1}(Y)X_{1}+...+f_{k}(Y)X_{k}$. Since $\mathcal{L}_{Y}(Z)$ is linear in $Y$, each $Y\mapsto f_{i}(Y)$ is a cochain. Let $Y,W\in\mathbb{R}^{k}$, it is immediate from the definition of $f_{i}$ that
\begin{align*}
& [W,[Y,Z]] = \left[W,\sum_{j = 1}^{d-1}f_{j}(Y)X_{j}\right] = \sum_{j = 1}^{d-1}Wf_{j}(Y)\cdot X_{j}, \\
& [Y,[W,Z]] = \left[Y,\sum_{j = 1}^{d-1}f_{j}(W)\cdot X_{j}\right] = \sum_{j = 1}^{d-1}Yu_{j}(W)\cdot X_{j}
\end{align*}
since $\mathbb{R}^{k}$ is abelian. The Jacobi identity implies
\begin{align*}
0 = [W,[Y,Z]] + [Z,[W,Y]] + [Y,[Z,W]] = [W,[Y,Z]] - [Y,[W,Z]]
\end{align*}
since $[W,Y] = 0$. It follows $Wu_{i}(Y) - Yu_{i}(W) = 0$, so each $f_{i}:\mathbb{R}^{k}\to C^{\infty}(M)$ is a cocycle (formula \ref{Eq:SecondCoboundaryOperator}). Since $\alpha$ is GH, Theorem \ref{Thm:DecompositionOfCochainSpaces} implies that each $f_{i}$ is the sum of a coboundary and a homomorphism. That is, for each $i = 1,...,k$ we find $c_{i}:\mathbb{R}^{k}\to\mathbb{R}$ and $v_{i}\in C^{\infty}(M)$ such that
\begin{align*}
[Y,Z] = \sum_{i = 1}^{d-1}(Yv_{i} + c_{i}(Y))\cdot X_{i},\quad Y\in\mathbb{R}^{k}.
\end{align*}
The lemma follows by redefining $Z$ as $\Tilde{Z} = Z - v_{1}\cdot X_{1} - ... - v_{k}\cdot X_{k}$.
\end{proof}
Denote the Lie algebra from Lemma \ref{L:ExistenceOfHomogeneous} by $\mathfrak{g}$, with associated simply connected Lie group $G$.
\begin{lemma}\label{L:StructureOfManifold}
The manifold $M$ can be written as $M = \Gamma\setminus G$ where $\Gamma\leq G$ is a uniform lattice and $\alpha$ acts $M$ by translations. Moreover, $G$ is a solvable group, and $\mathbb{R}^{k}$ embed in $G$ as a codimension $1$ abelian, normal subgroup $A$ such that $\alpha$ acts by translations of $A$.
\end{lemma}
\begin{proof}
The first part of the lemma is immediate since $\mathfrak{g}\to\Gamma(TM)$ defines a locally free action of $G$ on $M$, and the dimension of $G$ coincides with the dimension of $M$. Since this action of $G$ is locally free, the stabilizer of a point is a lattice, $\Gamma$, and since $M$ is compact $\Gamma$ is uniform. Let $\mathfrak{a}\subset\mathfrak{g}$ be the codimension $1$ abelian subalgebra generated by $X_{1},...,X_{k}$. By (the proof of) the previous lemma $[Z,\mathfrak{a}]\subset\mathfrak{a}$, so $\mathfrak{a}$ is an ideal in $\mathfrak{g}$. Since $\mathfrak{a}$ is a codimension $1$ abelian ideal, it follows that $\mathfrak{g}$ is solvable so $G$ is solvable. The last part of the lemma is clear if we let $A$ be the subgroup corresponding to $\mathfrak{a}\subset\mathfrak{g}$.
\end{proof}
Recall that the \textit{nilradical} of a group $G$ is the maximal normal nilpotent subgroup of $G$. We will use the following result by Mostow \cite[Section 5]{Mostow1954}.
\begin{theorem}[Mostow]\label{Thm:MostowPropertyOfLattice}
Let $\Gamma$ be a lattice in a connected solvable Lie group $G$ with nilradical $N$. Then $\Gamma\cap N$ is a lattice in $N$.
\end{theorem}
We can now finish the proof of Theorem \ref{MainThm:ThmB} $(iii)$.
\begin{proof}[Proof of Theorem \ref{MainThm:ThmB} $(iii)$]
Let $G$ be the group from Lemma \ref{L:StructureOfManifold}. If $G$ is not nilpotent, then $A$ would correspond to the nilradical (since it has codimension one, and is abelian so in particular nilpotent). This would then imply (Theorem \ref{Thm:MostowPropertyOfLattice}) that $\Gamma\cap A$ is a lattice in $A$, so the $A-$orbits would be compact. This is a contradiction since the translation action by $A$ (which coincides with the action $\alpha$) is minimal (Remark \ref{Rmk:MinimalityRemark}). It follows that $G$ is nilpotent. The theorem follows by Lemma \ref{L:RationalityOfQuasiAbelianGroups} since $G$ admits a codimension $1$ abelian normal subgroup that is not rational.
\end{proof}

\section{A nilmanifold without globally hypoelliptic action}

In this section, we produce a family of examples of compact nilmanifolds that do not admit any GH (abelian) actions. This should be contrasted with the results of \cite{Sandfeldt2024}, where many examples of GH actions are produced.
\begin{lemma}\label{L:RationalityOfQuasiAbelianGroups}
Let $G$ be a simply connected $\ell-$step quasi-abelian nilpotent Lie group with a codimension $1$ normal abelian subgroup $A$ and with a lattice $\Gamma$. If $G\neq H^{1}\times\mathbb{R}^{N}$ then $A$ is a rational subgroup.
\end{lemma}
\begin{proof}
We begin by proving that $A$ is a rational subgroup when $\ell\geq 3$. Let $\mathfrak{a}$ be the ideal associated with $A$. If $X\in\mathfrak{g}\setminus\mathfrak{a}$ then any $Y\in\mathfrak{g}$ can be written as $Y = \alpha_{Y}X + Y_{a}$ with $Y_{a}\in\mathfrak{a}$. Any bracket can now be written:
\begin{align}
[Y,W] = [\alpha_{Y}X + Y_{a},\alpha_{W}X + W_{a}] = \alpha_{Y}[X,W_{a}] + \alpha_{W}[Y_{a},X]\in\mathfrak{a}
\end{align}
where the inclusion follows since $\mathfrak{a}$ is an ideal. That is, we have $\mathfrak{g}_{(2)} = [\mathfrak{g},\mathfrak{g}]\leq\mathfrak{a}$. If $\ell\geq3$ then there is a rational vector $Y\in\mathfrak{g}_{(2)}$ such that $\ker{\rm ad}_{Y}\neq\mathfrak{g}$. On the other hand, since $\mathfrak{g}_{(2)}\leq\mathfrak{a}$ and $\mathfrak{a}$ is abelian, we have $\mathfrak{a}\leq\ker{\rm ad}_{Y}$. It follows that $\mathfrak{a} = \ker{\rm ad}_{Y}$ since $\dim(\ker{\rm ad}_{Y})\leq\dim(\mathfrak{g})-1 = \dim(\mathfrak{a})$. Since $Y$ is rational, $\ker{\rm ad}_{Y}$ is a rational subspace, so $\mathfrak{a}$ is a rational ideal.

Assume instead that $\ell = 2$. Since both the center of $\mathfrak{g}$ and $\mathfrak{g}_{(2)} = [\mathfrak{g},\mathfrak{g}]$ are rational subgroups we can write $\mathfrak{g} = \Hat{\mathfrak{g}}\oplus\mathbb{R}^{N}$ where the center of $\Hat{\mathfrak{g}}$ coincide with $\mathfrak{g}_{(2)}$. After dropping to $\Hat{\mathfrak{g}}$ we may assume without loss of generality that the center of $\mathfrak{g}$ coincides with $\mathfrak{g}_{(2)}$. Let $Z_{1},...,Z_{n}\in\mathfrak{g}_{(2)}$ be a rational basis, $X\in\mathfrak{g}\setminus\mathfrak{a}$ be a rational vector, $V'$ a rational complement of $\mathfrak{g}_{(2)}$ in $\mathfrak{g}$, and $V = V'\cap\mathfrak{a}$. If $W\in V$ and $[X,W] = 0$ then $W$ commute with both $X$ and $\mathfrak{a}$ so $W$ is central, which implies that $W = 0$ since $V\cap\mathfrak{g}_{(2)} = 0$. It follows that there is a basis $Y_{1},...,Y_{n}\in V$ such that $[X,Y_{j}] = Z_{j}$. If $U_{j}\subset V'$ is defined by: $Y\in U_{j}$ if $[X,Y]\in\mathbb{R}Z_{j}$, then $U_{j}$ is a rational subspace since $X,Z_{j}$ and $V'$ are rational. On the other hand, if $Y\in U_{j}$ then $Y = \alpha X + \alpha_{1}Y_{1}+...+\alpha_{n}Y_{n}$ and $[X,Y] = \alpha_{1}Z_{1}+...+\alpha_{n}Z_{n}\in\mathbb{R}Z_{j}$ so $Y = \alpha X + \alpha_{j}Y_{j}$. That is, we have an equality $U_{j} = \mathbb{R}X\oplus\mathbb{R}Y_{j}$. Let $\xi_{j}X + \eta_{j}Y_{j}$ be rational, then:
\begin{align}
[X,\xi_{j}X + \eta_{j}Y_{j}] = \eta_{j}Z_{j}\in\mathfrak{g}_{\mathbb{Q}}
\end{align}
so, since $Z_{j}$ is rational, $\eta_{j}\in\mathbb{Q}$. If we choose the vector $\xi_{j}X + \eta_{j}Y_{j}$ such that $\xi_{j},\eta_{j}\neq0$ then $(\xi_{j}X + \eta_{j}Y_{j})/\eta_{j} = \omega_{j}X + Y_{j}\in\mathfrak{g}_{\mathbb{Q}}$, so for each $j = 1,...,n$ there is a real number $\omega_{j}$ such that $\omega_{j}X + Y_{j}\in\mathfrak{g}_{\mathbb{Q}}$. If we take the bracket between $\omega_{i}X + Y_{i}$ and $\omega_{j}X + Y_{j}$ with $i\neq j$ then:
\begin{align}
[\omega_{i}X + Y_{i},\omega_{j}X + Y_{j}] = \omega_{i}Z_{j} - \omega_{j}Z_{i}\in\mathfrak{g}_{\mathbb{Q}}.
\end{align}
All vectors $Z_{i}$, $Z_{j}$, and $[\omega_{i}X + Y_{i},\omega_{j}X + Y_{j}]$ are rational, so $\omega_{i}$ and $\omega_{j}$ are both rational numbers. It follows that $Y_{1},...,Y_{n}\in V$ are all rational so $\mathfrak{a}$ is rational (since $Y_{1},...,Y_{n},Z_{1},...,Z_{n}$ form a rational basis of $\mathfrak{a}$). If $\mathfrak{a}$ is not a rational subspace, then $n = 1$ so $\mathfrak{g} = {\rm span}(X,Y_{1},Z_{1})\cong\mathfrak{h}^{1}$.
\end{proof}
\begin{lemma}\label{L:IrrationalInDerivedSubgroup}
Let $G$ be a $2-$step nilpotent Lie group with lattice $\Gamma\leq G$, $\rho:\mathbb{R}^{k}\to G$ a homomorphism, and $\alpha(\mathbf{t})x = x\rho(\mathbf{t})$ the corresponding action. If $\alpha$ is GH then ${\rm Im}(\rho)\cap[G,G]$ is irrational in $[G,G]$\footnote{i.e., the only rational subgroup which contain ${\rm Im}(\rho)\cap[G,G]$ is $[G,G]$.}.
\end{lemma}
\begin{proof}
Suppose, for contradiction, that there is a proper rational subgroup $W\leq[G,G]$ such that ${\rm Im}(\rho)\cap[G,G]\subset W$. We may assume, without loss of generality, that $\dim(W) = \dim([G,G])-1$. Let $N' = G/W$, and let $Z$ be a rational subgroup of the center of $N'$ that is complementary to $[N',N']$. Finally, define $N = N'/Z$, so that $N$ is $2-$step with $1-$dimensional center and the center of $N$ coincide with $[N,N]$. This implies that $N\cong H^{g}$ for some $g$ (Lemma \ref{L:ClassificationOf2StepNilpotent}). Since $G\to N$ is defined over $\mathbb{Q}$ we find a lattice $\Lambda\leq N$ such that the image of $\Gamma$ under the quotient map $G\to N$ maps $\Gamma$ onto $\Lambda$. This defines a submersion:
\begin{align}
M_{\Gamma}\to N_{\Lambda}.
\end{align}
Let $\Hat{\rho}:\mathbb{R}^{k}\to N$ be the map defined by $\mathbb{R}^{k}\xrightarrow{\rho} G\to N$, and let $\beta$ be the corresponding translation action on $N_{\Lambda}$. By Lemma \ref{L:PropertiesOrbitLaplacian2} the action $\beta$ that is GH. Lemma \ref{L:GHonHeisenberg} implies that ${\rm Im}(\Hat{\rho})\cap [N,N] = [N,N]$ which is a contradiction since we assumed that ${\rm Im}(\rho)\cap[G,G]\subset W$, and $W$ is mapped to $0$ under the map $G\to N$.
\end{proof}
\begin{theorem}
Let $\ell\geq3$ and $G$ be a $\ell-$step quasi-abelian nilpotent Lie group with a lattice $\Gamma$, then $M_{\Gamma}$ supports no {\rm GH} action.
\end{theorem}
\begin{proof}
Assume for contradiction that $\alpha$ is a GH action on $M_{\Gamma}$. By Theorem \ref{MainThm:ThmB}, case $(i)$ there is no loss of generality to assume that the GH action $\alpha$ is by translations. We write $\alpha(\mathbf{t})x = x\rho(\mathbf{t})$ for some homomorphism $\rho:\mathbb{R}^{k}\to G$, and to simplify the notation we write $H = {\rm Im}(\rho)\leq G$. Let $A\leq G$ be the normal codimension $1$ subgroup and let $X\in\mathfrak{g}_{\mathbb{Q}}$ be a rational vector complementary to $\mathfrak{a}$. Since $\mathfrak{a}$ is rational (Lemma \ref{L:RationalityOfQuasiAbelianGroups}) and $\alpha$ acts minimally on $M_{\Gamma}$ we have $\mathfrak{h}\not\leq\mathfrak{a}$ so there is $Y\in\mathfrak{h}$ that can be written $Y = \lambda X + Y_{a}$ with $\lambda\neq0$. Let:
\begin{align}
\mathfrak{p} = \{W\in\mathfrak{a}\text{ : }[Y,W] = 0\} = \{W\in\mathfrak{a}\text{ : }[X,W] = 0\}
\end{align}
then, since $\mathfrak{a}$ and $X$ are rational, $\mathfrak{p}$ is a rational subspace of $\mathfrak{a}$. We claim $\mathfrak{g}_{(2)}\leq\mathfrak{p}$, which is a contradiction since $[X,\mathfrak{g}_{(2)}] = [\mathfrak{g},\mathfrak{g}_{(2)}] = \mathfrak{g}_{(3)}\neq0$. Let $N = G/G_{(3)}$ and $\Lambda$ be the image of $\Gamma$ under the map $G\to N$. We write $H' = H/G_{(3)}$ and $\mathfrak{p}'$ as the image of $\mathfrak{p}$ in $\mathfrak{n}$. By Lemma \ref{L:IrrationalInDerivedSubgroup} $H'\cap[N,N]$ is irrational, on the other hand we have $\mathfrak{h}'\cap[\mathfrak{n},\mathfrak{n}]\subset\mathfrak{p}'$ since any element of $\mathfrak{h}$ commutes with $Y$ (since $\alpha$ is an abelian action). Since $\mathfrak{p}'\cap[\mathfrak{n},\mathfrak{n}]$ is rational and contains $\mathfrak{h}'\cap[\mathfrak{n},\mathfrak{n}]$ it follows that $\mathfrak{p}'\cap[\mathfrak{n},\mathfrak{n}] = [\mathfrak{n},\mathfrak{n}]$. That is, $\mathfrak{p}\cap\mathfrak{g}_{(2)}\to\mathfrak{g}_{(2)}/\mathfrak{g}_{(3)}$ is surjective. Let $V\leq\mathfrak{p}$ be any subspace such that $V\to\mathfrak{g}_{(2)}/\mathfrak{g}_{(3)}$ is surjective, then:
\begin{align}
\mathfrak{g}_{(2)} = V + [\mathfrak{g},V] + [\mathfrak{g},[\mathfrak{g},V]] +... = V + [X,V] + [X,[X,V]] +...
\end{align}
Since $V\leq\mathfrak{p}$ and $[X,\mathfrak{p}] = 0$ (by definition) this implies that $\mathfrak{g}_{(2)} = V\subset\mathfrak{p}$.
\end{proof}

\appendix

\section{Proof of Lemma 2.2}
\label{Appendix:A}

In this appendix, we prove Lemma \ref{L:BackgroundTameEstimates2}. Let $L:C^{\infty}(M)\to C^{\infty}(M)$ be an operator, $\mu$ a volume form on $M$ and $L^{*}$ an adjoint of $L$, $L'(\overline{f}\mu) = \overline{L^{*}f}\cdot\mu$. Denote by $\norm{\cdot}_{n}$ the $n$th Sobolev norm and $W^{n}(M,\mu)$ the $n$th Sobolev space (using $L^{2}-$norms).
\begin{proof}[Proof of Lemma \ref{L:BackgroundTameEstimates2}]
Assume that $L^{*}$ is GH. We begin by showing $\dim\ker L < \infty$. There is $s_{0}\in\mathbb{N}_{0}$ such that $\norm{Lu}_{0}\leq c\norm{u}_{s}$ for $s\geq s_{0}$ (since $L$ is continuous in the Fréchet topology). That is, we can extend $L$ to $L:W^{s}(M,\mu)\to L^{2}(M,\mu)$. If $u\in\ker L\subset W^{s}(M,\mu)$ and $f\in C^{\infty}(M)$ then
\begin{align*}
0 = \int_{M}Lu\cdot\overline{f}\intd\mu = \int_{M}u\cdot\overline{L^{*}f}\intd\mu,
\end{align*}
so $Lu = (L^{*})'u = 0$ when we consider $u\in\mathcal{D}'(M)$. The function $0$ is smooth, so $u\in C^{\infty}(M)$ by Lemma \ref{L:BackgroundTameEstimates1}. It follows that the kernel of $L:W^{s_{0}}(M,\mu)\to L^{2}(M,\mu)$ coincide with the kernel of $L:W^{s_{0}+1}(M,\mu)\to L^{2}(M,\mu)$. The inclusion $W^{s_{0}+1}(M,\mu)\to W^{s_{0}}(M,\mu)$ is compact \cite[Theorem 3.6]{Hebey1996}, so the identity map on $\ker L\subset W^{s_{0}}(M,\mu)$ is compact, and since $\ker L$ is a Banach space it follows that $\ker L$ is finite dimensional.

Consider the map
\begin{align}\label{Eq:ExtensionOfGHoperator}
L:\text{Dom}(L)\subset W^{-1}(M,\mu)\to C^{\infty}(M)
\end{align}
where the domain $\text{Dom}(L)$ of $L$ coincide with $C^{\infty}(M)\subset W^{-1}(M,\mu)$ ($L^{*}$ is GH so if $LD\in C^{\infty}(M)$ then $D\in C^{\infty}(M)$ by Lemma \ref{L:BackgroundTameEstimates1}). We claim that $L$, in Equation \ref{Eq:ExtensionOfGHoperator}, is a closed operator. Indeed, if $(D_{n},LD_{n})\in\text{Graph}(L)$ converges to some $(D,g)$ then
\begin{align*}
LD(f) = D(L^{*}f) = \lim_{n\to\infty}D_{n}(L^{*}f) = \lim_{n\to\infty}LD_{n}(f) = \int_{M}f\cdot\overline{g}\intd\mu
\end{align*}
so $LD = g$ and $(D,g)\in\text{Graph}(L)$. Define the map $T:\text{Graph}(L)\to L^{2}(M,\mu)$ by $(D,LD)\mapsto D$, which makes sense since $D\in\text{Dom}(L) = C^{\infty}(M)$. We claim that the operator $T$ is closed from the Fréchet space $\text{Graph}(L)$ to $L^{2}(M,\mu)$. Assume that $((D_{n},LD_{n}),D_{n})\in\text{Graph}(T)$ converges to some $((D,LD),g)$. Each $D_{n}$ is represented by some $g_{n}\in C^{\infty}(M)$:
\begin{align*}
D_{n}(f) = \int_{M}f\cdot\overline{g}_{n}\intd\mu,
\end{align*}
and since $T(D_{n},LD_{n}) = g_{n}$ it follows that $g_{n}\to g$ weakly in $L^{2}(M,\mu)$. If $g_{n}\to g$ weakly in $L^{2}(M,\mu)$, then, since $L^{2}(M,\mu)\to W^{-1}(M,\mu)$ is compact, it follows that $D_{n}\to g$ strongly $W^{-1}(M,\mu)$. Finally since $D_{n}\to D$ it follows that $D = g$, so $((D,LD),g) = ((g,Lg),g)\in\text{Graph}(T)$. By the Closed graph theorem in Fréchet spaces, $T$ is a continuous operator. The norms $\norm{\cdot}_{n}$ are increasing, so we find some integer $s\in\mathbb{N}_{0}$
\begin{align}\label{Eq:PreliminaryGraphProjectionEstimate}
\norm{f}_{0}\leq C\left(\norm{f}_{-1} + \norm{Lf}_{s}\right),\quad (f,Lf)\in{\rm Graph}(L).
\end{align}
Let $e_{1},...,e_{N}\in\ker L\subset C^{\infty}(M)$ be a ON-basis of $\ker L$. Define $V\subset C^{\infty}(M)$ by
\begin{align*}
V = \left\{v\in C^{\infty}(M)\text{ : }\int_{M}v\cdot\overline{e}_{j}\intd\mu = 0\right\}.
\end{align*}
We claim that for $v\in V$ we have $\norm{v}_{0}\leq C'\norm{Lv}_{s}$. Assume for contradiction that this is not the case. There exists $v_{j}\in V$ such that $\norm{v_{j}}_{0} = 1$ and $\norm{Lv_{j}}_{s}\to 0$. Since $\norm{v_{j}}_{0} = 1$ we may assume, after dropping to a subsequence, that $v_{j}\to v$ weakly in $L^{2}(M,\mu)$. For $f\in C^{\infty}(M)$
\begin{align*}
\int_{M}v\cdot\overline{L^{*}f}\intd\mu = \lim_{j\to\infty}\int_{M}v_{j}\cdot\overline{L^{*}f}\intd\mu = \lim_{j\to\infty}\int_{M}Lv_{j}\cdot\overline{f}\intd\mu = 0,
\end{align*}
so $Lv = 0$ in the sense of distributions. Since $L^{*}$ is GH it follows that $v\in C^{\infty}(M)$ and $Lv = 0$ in the strong sense. Since $v_{j}\in V$ converges weakly to $v$ we have
\begin{align*}
\int_{M}v\cdot\overline{e}_{n}\intd\mu = \lim_{j\to\infty}\int_{M}v_{j}\cdot\overline{e}_{n}\intd\mu = 0,\quad n = 1,...,N,
\end{align*}
where we have used the definition of $V$. That is, $v\in\ker L\cap V$ which implies $v = 0$. Since this holds for any weak limit point of $(v_{j})_{j}$ it follows that $v_{j}\to 0$ weakly in $L^{2}(M,\mu)$. The embedding $L^{2}(M,\mu)\to W^{-1}(M,\mu)$ is compact, so $v_{j}\to 0$ holds with respect to the norm topology in $W^{-1}(M,\mu)$. Equation \ref{Eq:PreliminaryGraphProjectionEstimate} now implies
\begin{align*}
1 = \lim_{j\to\infty}\norm{v_{j}}_{0}\leq\lim_{j\to\infty}C\left(\norm{v_{j}}_{-1} + \norm{Lv_{j}}_{s}\right) = 0,
\end{align*}
which is a contradiction. It follows that $\norm{v}_{0}\leq C'\norm{Lv}_{s}$ for $v\in V$. Assume that $Lv_{j}\to f$ in $C^{\infty}(M)$. Since $V$ is complementary to $\ker L$ we may assume that $v_{j}\in V$. It follows that
\begin{align*}
\norm{v_{j}}_{0}\leq C'\norm{Lv_{j}}_{s}
\end{align*}
so, after possibly dropping to a subsequence, we may assume that $v_{j}\to v$ weakly in $L^{2}(M,\mu)$. For $g\in C^{\infty}(M)$
\begin{align*}
\int_{M}L^{*}g\cdot\overline{v}\intd\mu = \lim_{j\to\infty}\int_{M}L^{*}g\cdot\overline{v}_{j}\intd\mu = \lim_{j\to\infty}\int_{M}g\cdot\overline{Lv_{j}}\intd\mu = \int_{M}g\cdot\overline{f}\intd\mu
\end{align*}
so $Lv = f$ in the sense of distributions. Since $f\in C^{\infty}(M)$ it follows by Lemma \ref{L:BackgroundTameEstimates1} that $v\in C^{\infty}(M)$, and $Lv = f$ strongly. That is, $L$ has closed image.

The last claim is standard. Suppose that $L$ is also GH. Since $(L^{*})^{*} = L$ it follows that $\dim\ker L^{*} < \infty$. Let $f_{1},...,f_{M}\in\ker L^{*}$ be a ON-basis. Define
\begin{align*}
U = \left\{u\in C^{\infty}(M)\text{ : }\int_{M}u\cdot\overline{f}_{j}\intd\mu = 0,\text{ }j=1,...,M\right\}.
\end{align*}
We claim that $\text{Im}(L) = U$, this proves the lemma since $U$ is complementary to $\ker L^{*}$. For any $v\in C^{\infty}(M)$ we have
\begin{align*}
\int_{M}Lv\cdot\overline{f}_{j}\intd\mu = \int_{M}v\cdot\overline{L^{*}f_{j}}\intd\mu = 0,\quad j = 1,...,M,
\end{align*}
so ${\rm Im}(L)\subset U$. To show the converse, it suffices to show that $\text{Im}(L)$ is dense in $U$ since $L$ has closed image. If ${\rm Im}(L)$ is not dense in $U$, there exist, by the Hahn-Banach theorem, a distribution $D$ such that $D|_{U}\neq 0$, $D(Lu) = 0$ for all $u\in C^{\infty}(M)$ (equivalently $L'D = 0$), and $D|_{\ker L^{*}} = 0$. Since $L'D = L^{*}D = 0$, and $L$ is assumed to be GH it follows that $D = g\in C^{\infty}(M)$ (by Lemma \ref{L:BackgroundTameEstimates1}). It follows that $L^{*}g = 0$, or $g\in\ker L^{*}$. Since $D|_{\ker L^{*}} = 0$ we have $D(g) = 0$, but $D = g$ in the sense of distributions so $0 = D(g) = \norm{g}_{0}^{2}$ which is a contradiction since $D|_{U}\neq0$. It follows that $\text{Im}(L)$ is dense in $U$.
\end{proof}

%\newpage
\bibliography{main.bib}{}

\begin{thebibliography}{10}

\bibitem{ChenChiEqGW}
W.~Chen and M.~Chi.
\newblock Hypoelliptic vector fields and almost periodic motions on the torus $\mathbb{T}^{n}$.
\newblock {\em Partial Differential Equations}, 25(1-2):337–354, 2000.

\bibitem{CorwinGreenleaf}
L.~Corwin and F.~P. Greenleaf.
\newblock {\em Representations of Nilpotent Lie Groups and Their Applications: Volume 1, Part 1, Basic Theory and Examples}.
\newblock Cambridge University Press, 1990.

\bibitem{CyganRichardson1988}
J.~Cygan and L.~Richardson.
\newblock Globally hypoelliptic systems of vector fields on nilmanifolds.
\newblock {\em Journal of functional analysis}, 77:364--371, 1988.

\bibitem{CyganRichardson1991}
J.~Cygan and L.~Richardson.
\newblock D-harmonic distributions and global hypoellipticity on nilmanifolds.
\newblock {\em Pacific journal of mathematics}, 147(1):29--46, 1991.

\bibitem{DanijelaGH1}
D.~Damjanovic.
\newblock Actions with globally hypoelliptic leafwise laplacian and rigidity.
\newblock {\em Journal d Analyse Mathématique}, 129(1):139–163, 2016.

\bibitem{DanijelaGH2}
D.~Damjanovic, J.~Tannis, and Z.~J. Wang.
\newblock On globally hypoelliptic abelian actions and their existence on homogeneous spaces.
\newblock {\em Discrete and Continuous Dynamical Systems}, 40(12):6747--6766, 2020.

\bibitem{DamjanovicFayad2019}
D.~Damjanović and B.~Fayad.
\newblock On local rigidity of partially hyperbolic affine $\mathbb{Z}^{k}$ actions.
\newblock {\em Reine Angew. Math.}, 751:1–26, 2019.

\bibitem{DamjanovicFayadSaprykina2023}
D.~Damjanović, B.~Fayad, and M.~Saprykina.
\newblock Kam-rigidity for parabolic affine abelian actions.
\newblock {\em arXiv:2303.04367}, 2023.

\bibitem{DamjanovicKatok2010}
D.~Damjanović and A.~Katok.
\newblock Local rigidity of partially hyperbolic actions i. kam method and $\mathbb{Z}^{k}$ actions on the torus.
\newblock {\em Annals of Mathematics}, 172(3):1805–58, 2010.

\bibitem{DamjanovicKatok2011Parabolic}
D.~Damjanović and A.~Katok.
\newblock Local rigidity of homogeneous parabolic actions: I. a model case.
\newblock {\em Journal of Modern Dynamics}, 5(2):203--235, 2011.

\bibitem{HarmonicAnalysis}
A.~Deitmar and S.~Echterhoff.
\newblock {\em Principles of Harmonic Analysis}.
\newblock Springer, 2014.

\bibitem{FisherKalininSpatzier2013}
D.~Fisher, B.~Kalinin, R.~J. Spatzier, and with an appendix~by J.~Davis.
\newblock Global rigidity of higher rank anosov actions on tori and nilmanifolds.
\newblock {\em J. Amer. Math. Soc.}, 26(1):167–198, 2013.

\bibitem{FlaminioForni2003}
L.~Flaminio and G.~Forni.
\newblock Invariant distributions and time averages for horocycle flows.
\newblock {\em Duke Mathematical Journal}, 119(3):465--526, 2003.

\bibitem{FlaminioForni2006}
L.~Flaminio and G.~Forni.
\newblock Equidistribution of nilflows and applications to theta sums.
\newblock {\em Ergodic Theory and Dynamical Systems}, 26(2):409--433, 2006.

\bibitem{ForniCohEq}
L.~Flaminio and G.~Forni.
\newblock On the cohomological equation for nilflows.
\newblock {\em Journal of Modern Dynamics}, 1(1):37--60, 2007.

\bibitem{FlaminioForni2023}
L.~Flaminio and G.~Forni.
\newblock Equidistribution of nilflows and bounds on weyl sums.
\newblock {\em arXiv:2302.03618}, 2023.

\bibitem{ForniRHGWC}
L.~Flaminio, G.~Forni, and F.~R. Hertz.
\newblock Invariant distributions for homogeneous flows and affine transformations.
\newblock {\em Journal of Modern Dynamics}, 10(2):33--79, 2016.

\bibitem{FlaminioPaternainEmbedding}
L.~Flaminio and M.~Paternain.
\newblock Linearization of cohomology-free vector fields.
\newblock {\em Discrete and Continuous Dynamical Systems}, 29(3):1031--1039, 2011.

\bibitem{ForniGLWdim3}
G.~Forni.
\newblock On the greenfield-wallach and katok conjectures.
\newblock {\em arXiv:0706.3981}, 2007.

\bibitem{Franks1969}
J.~Franks.
\newblock Anosov diffeomorphisms on tori.
\newblock {\em Transactions of the American Mathematical Society}, 145:117--24, 1969.

\bibitem{GreenfieldWallach1}
S.~J. Greenfield and N.~R. Wallach.
\newblock Globally hypoelliptic vector fields.
\newblock {\em Topology}, 12:247–254, 1973.

\bibitem{GreenfieldWallach2}
S.~J. Greenfield and N.~R. Wallach.
\newblock Remarks on global hypoellipticity.
\newblock {\em Transactions of the American Mathematical Society}, 183:153--164, 1973.

\bibitem{Hamilton}
R.~S. Hamilton.
\newblock The inverse function theorem of nash and moser.
\newblock {\em Bull. Amer. Math. Soc.}, 7(1):65--222, 1982.

\bibitem{Hebey1996}
E.~Hebey.
\newblock {\em Sobolev Spaces on Riemannian Manifolds}.
\newblock Springer-Verlag, Heidelberg, first edition, 1996.

\bibitem{RordriguezHertzRodriguezHertz2006}
F.~R. Hertz and J.~R. Hertz.
\newblock Cohomology free systems and the first betti number.
\newblock {\em Discrete and Continuous Dynamical Systems}, 15(1):193–196, 2006.

\bibitem{Hurder}
S.~Hurder.
\newblock Problems on rigidity of group actions and cocycles.
\newblock {\em Ergodic Theory and Dynamical Systems}, 5:473--484, 1985.

\bibitem{KatokCombinatorial}
A.~Katok.
\newblock {\em Combinatorial constructions in ergodic theory and dynamics}, volume~30.
\newblock American Mathematical Society, Providence, RI, 2003.

\bibitem{Kim2022}
M.~Kim.
\newblock Effective equidistribution for generalized higher-step nilflows.
\newblock {\em Ergodic Theory and Dynamical Systems}, 42(12):3656–3715, 2022.

\bibitem{Kocsard2007}
A.~Kocsard.
\newblock {\em Toward the classification of cohomology-free vector fields}.
\newblock PhD thesis, IMPA, Rio de Janeiro, 2007.
\newblock arXiv:0706.4053.

\bibitem{KocsardGLWdim3}
A.~Kocsard.
\newblock Cohomologically rigid vector fields : the katok conjecture in dimension 3.
\newblock {\em Annales De L Institut Henri Poincare-analyse Non Lineaire}, 26:1165--1182, 2009.

\bibitem{Manning1974}
A.~Manning.
\newblock There are no new anosov diffeomorphisms on tori.
\newblock {\em American Journal of Mathematics}, 96(3):422–29, 1974.

\bibitem{Mostow1954}
G.~D. Mostow.
\newblock Factor spaces of solvable groups.
\newblock {\em Annals of Mathematics}, 60(1):1--27, 1954.

\bibitem{Mukherjee2015}
A.~Mukherjee.
\newblock {\em Differential Topology}.
\newblock Birkhäuser Verlag, Basel, 2nd edition, 2015.

\bibitem{Rodriguez-HertzWang2014}
F.~{Rodriguez Hertz} and Z.~Wang.
\newblock Global rigidity of higher rank abelian anosov algebraic actions.
\newblock {\em Invent. math.}, 198:165–209, 2014.

\bibitem{Sandfeldt2024}
S.~Sandfeldt.
\newblock Classification of abelian actions with globally hypoelliptic orbitwise laplacian ii: Examples of globally hypoelliptic actions and local rigidity.
\newblock {\em In preparation}, 2024.

\bibitem{dosSantosLuz1998}
N.~M.~D. Santos and R.~U. Luz.
\newblock Cohomology-free diffeomorphisms of low-dimension tori.
\newblock {\em Ergodic Theory and Dynamical Systems}, 18(4):985--1006, 1998.

\bibitem{TaubesWeinstein}
C.~H. Taubes.
\newblock The seiberg–witten equations and the weinstein conjecture.
\newblock {\em Geom. Topol.}, 11(4):2117--2202, 2007.

\bibitem{Wang2019}
Z.~Wang.
\newblock Local rigidity of parabolic algebraic actions.
\newblock {\em arXiv:1908.10496}, 2019.

\bibitem{Wang2022}
Z.~Wang.
\newblock Local rigidity of higher rank partially hyperbolic algebraic actions.
\newblock {\em arXiv:1103.3077}, 2022.

\end{thebibliography}
\bibliographystyle{abbrv}

\end{document}